\documentclass{amsart}
\usepackage{amsmath}
\usepackage{stackrel}
\usepackage{amssymb, MnSymbol}
\usepackage{color}
\usepackage{amscd}
\usepackage[english]{babel}
\usepackage{amsmath}
\usepackage{amsthm}
\usepackage{amssymb}
\usepackage{amscd}
\usepackage[all]{xy}
\usepackage{color}
\usepackage[latin1]{inputenc}
\newtheorem{teor}{Theorem}[section]
\newtheorem*{defi*}{Definition}
\newtheorem*{teorA*}{Theorem A}
\newtheorem*{teorB*}{Theorem B}
\newtheorem*{teorC*}{Theorem C}
\newtheorem{prop-defi}{Proposition-Definition} 
\newtheorem{lema}[teor]{Lemma}
\newtheorem{prop}[teor]{Proposition}
\newtheorem{cor}[teor]{Corollary}
\newtheorem{rem}[teor]{Remark}
\newtheorem{rems}[teor]{Remarks}
\newtheorem{ejem}[teor]{Example}
\newtheorem{ejems}[teor]{Examples}

\newtheorem{setting}[teor]{Setting}
\newtheorem{Obs}[teor]{Observation}
\newtheorem{Com}[teor]{Comment}

\def\Hom{\mathop{\rm Hom}\nolimits}

\def\mod {\mathop{\rm mod}\nolimits}
\def\End {\mathop{\rm End}\nolimits}

\def\Ext {\mathop{\rm Ext}\nolimits}
\def\Soc {\mathop{\rm Soc}\nolimits}
\def\Ker {\mathop{\rm Ker}\nolimits}
\def\Coker {\mathop{\rm Coker}\nolimits}
\def\Im {\mathop{\rm Im}\nolimits}

\def\ann {\mathop{\rm ann}\nolimits}
\def\add {\mathop{\rm add}\nolimits}
\def\pdim {\mathop{\rm pdim}\nolimits} 
\def\gldim {\mathop{\rm gldim}\nolimits} 
\def\top {\mathop{\rm top}\nolimits}
\def\id {\mathop{\rm id}\nolimits}
\def\soc {\mathop{\rm soc}\nolimits} 
\def\op {\mathop{\rm op}\nolimits}
\def\lfindim {\mathop{\rm l.findim }\nolimits}
\def\rfindim {\mathop{\rm r.findim }\nolimits}

\def\la{\Lambda}
\def\pinf{\mathcal{P}^{<\infty}}
\def\XX {{\mathbb X}}
\def\latilde{\widetilde{\Lambda}}
\def\etilde {\widetilde{e}}
\def\Jtilde {\widetilde{J}}

\newcommand\lamod{\Lambda\mbox{\rm-mod}}
\newcommand\laMod{\Lambda\mbox{\rm-Mod}}
\newcommand\elaemod{e\Lambda e\mbox{\rm-mod}}
\newcommand\elaeMod{e\Lambda e\mbox{\rm-Mod}}
\newcommand\laprimemod{\Lambda'\mbox{\rm-mod}}
\newcommand\latiltilmod{\widetilde{\latilde}\mbox{\rm-mod}}
\newcommand\Amod{A\mbox{\rm-mod}}

\newcommand\modla{\mbox{\rm mod-}\Lambda}

\newcommand\modelae{\mbox{\rm mod-}e\Lambda e}
\newcommand\modlatilde{\mbox{\rm mod-}\latilde}
\newcommand\modeprimelaeprime{\mbox{\rm mod-}e'\Lambda e'}
\newcommand\modetillatiletil{\mbox{\rm mod-}\etilde\latilde\etilde}

\newcommand\edge{\ar@{-}}
\newcommand\drbl{\save+<0ex,-2ex> \drop{\bullet} \restore}

\title[contravariant finiteness and iterated strong tilting]{Contravariant finiteness and iterated strong tilting}

\author[B.~Huisgen-Zimmermann]{Birge Huisgen-Zimmermann}
\address[Birge Huisgen-Zimmermann]{%
Department of Mathematics, University of California, Santa Barbara\\
CA 93106 \\
USA}
\email{birge@math.ucsb.edu}

\author[Z.~Nazemian]{Zahra Nazemian}
\address[Zahra Nazemian]{%
University of Graz \\
Heinrichstrasse 36 \\
8010 Graz, Austria}
\email{zahra.nazemian@uni-graz.at}

\author[M.~Saor\'\i n]{Manuel Saor\'\i n}
\address[Manuel Saor\'\i n]{%
Departamento de Matem\'aticas\\
Universidad de Murcia, Aptdo. 4021\\
30100 Espinardo, Murcia\\
SPAIN}
\email{msaorinc@um.es}%

\thanks{ The second named author would like to thank all members of Department of Mathematics at the University of Murcia for their warm hospitality during her stay as a post doc, supported by a grant of  the research project 19880/GERM/15 of the Fundaci\'on S\'eneca of Murcia.  The third named author has been supported by the  Grant PID2020-113206GB-I00 funded by MCIN/AEI/10.13039/501100011033
}

\begin{document}

\begin{abstract}
{ 
Let $\pinf (\lamod)$
 be the category of finitely generated left modules of finite projective dimension over a basic Artin algebra
$\Lambda$.  We 
develop a widely applicable criterion that reduces the test for contravariant finiteness of
 $\pinf (\lamod)$ in $\lamod$ to 
corner algebras
 $e \Lambda e$ for suitable idempotents $e \in \Lambda$.  
The reduction substantially facilitates access to the 
numerous homological benefits entailed by contravariant finiteness of 
$\pinf (\lamod)$.   
The consequences pursued 
here hinge on the fact that this finiteness condition is known to be equivalent to the existence of a strong tilting 
object in $\lamod$.  
We moreover characterize the situation in which the process of strongly tilting 
$\lamod$ allows for 
unlimited iteration:  This occurs precisely when, in the category  
 $\modlatilde$ of right modules over the strongly tilted algebra $\latilde$, the subcategory 
of modules of finite projective dimension is in turn contravariantly finite; the latter condition can, once again, be tested on 
suitable corners 
$e \Lambda e$ of 
the original algebra $\Lambda$.  In the (frequently occurring) positive case, the  sequence of 
consecutive strong tilts,  $\latilde$,  $\widetilde{\latilde}$,  $\widetilde{\widetilde{\latilde}}, \dots$, 
is shown to be periodic with period $2$ (up to 
Morita equivalence); moreover, any two adjacent categories in the sequence $\pinf (\modlatilde)$, 
$\pinf(\latiltilmod)$, $\pinf(\mbox{\rm mod-}\widetilde{\widetilde{\latilde}}), \dots$, alternating between right and left modules, are  dual via contravariant  $\Hom$-functors
 induced by tilting bimodules which are strong on both sides.
   
Our methods rely on comparisons of right $\pinf$-approximations in the categories
$\lamod$, $\elaemod$ and the Giraud 
subcategory of $\lamod$ determined by $e$; these interactions hold interest in their own right.  In particular, 
they underlie our analysis of the indecomposable direct summands of strong tilting modules. 
}
\end{abstract}

\maketitle

{\bf {2020} Mathematics Subject Classification: {16D90, 16G10, 16E05, 16E30, 16S90}}

\section{Introduction}

{\it Co- and contravariant finiteness\/} of a subcategory $\mathcal{A}$ of the category 
$\lamod$
 of finitely generated modules over an Artin algebra $\Lambda$ were first considered by Auslander and 
Smal\o\ in \cite{AS1} and \cite{AS2}: 
 If $\mathcal{A}= \add(\mathcal{A})$ 
is closed under extensions, the combination of these two finiteness conditions implies the existence of internal almost split sequences in 
$\mathcal{A}$.  
Here we focus on the full subcategory
 $\mathcal{A} = \pinf (\lamod)$ 
consisting of the modules of finite projective dimension; in this situation contravariant finiteness implies  the dual property, covariant finiteness \cite {BirgeSmalo}.  (See Section 2 for notation and definitions of the italicized terms.)  Subsequently, it was found that contravariant finiteness of 
$\pinf (\lamod)$
 entails a plethora of additional homological benefits for 
$\lamod$ 
and the unrestricted module category
 $\laMod$.  
Namely, this finiteness condition not only validates the finitistic dimension conjectures for left 
$\Lambda$-modules, i.e., 
confirms that 
$\lfindim( \Lambda)  = \text{l.Findim}(\Lambda)  < \infty$ 
in that case, but gives rise to an intrinsic description of the 
$\Lambda$-modules 
of finite projective dimension; see \cite{AR1} and \cite{BirgeSmalo}.  

Crucial to our present investigation are the following facts: 
$\pinf (\lamod)$ 
is contravariantly finite in
$\lamod$ 
if and only if
 $\lamod$
 contains a {\it strong tilting module\/},  
%(see \cite{AR1}), 
i.e., a tilting module $T$ which is relatively $\Ext$-injective within the category $\pinf(\lamod)$.  Such a module $T$ is alternatively characterized by the condition that the contravariant functor 
$\Hom_{\Lambda}(- , T)$
 induces a duality between 
$\pinf (\lamod)$ and 
a specifiable resolving subcategory of 
$\pinf(\modlatilde)$, 
where 
$\latilde = \End_{\Lambda} (T)^{\op}$ 
(see \cite{Birgedualities}  and \cite{HuisgenZimmermann-Saorin}).  In this situation we say that $\mod\text{-}\latilde$ results from $\lamod$ via strong tilting.  Finally, we recall that, up to isomorphism, 
$\lamod$ 
has at most one {\it basic\/} strong tilting module; see \cite{AR1}.

In the present article we tackle the foremost obstacle that stands in the way of applying the theory we sketched: namely, the notoriously difficult task of deciding whether, for specific 
choices of $\Lambda$, the category  $\pinf (\lamod) $ is contravariantly finite in $\lamod$.  The driving goals of our investigation are to significantly reduce this difficulty and to explore the possibility and effect of iterating the process of strongly tilting 
$\lamod$.
In both directions we make headway, next to retrieving known results in a simpler, more uniform format.  

In more detail:  In the first of our main results we address an Artin algebra $\Lambda$ for which
 $\pinf (\lamod)$ is known to be contravariantly finite in $\lamod$. 
 %and build on the status of $\pinf(\modlatilde)$ in the strongly tilted module category $\modlatilde$.  
In general, $\pinf(\modlatilde)$ need not inherit contravariant finiteness.  However, whenever it does, the process of strongly tilting $\lamod$ may be iterated in the following sense:  $\lamod$ {\it allows for unlimited iteration of strong tilting} if there exists an infinite sequence of basic algebras $(\la_i)_{i \ge 0}$ with $\la_0 = \la$ such that $\mod\text{-}\la_{i+1}$ results from $\la_i\text{-}\mod$ via strong tilting if $i$ is even, and $\la_{i+1}\text{-}\mod$ results from $\mod\text{-}\la_i$ via strong tilting if $i$ is odd.

%and $\la_1 = \latilde$, next to a sequence of basic strong tilting modules $(T_i)_{i \ge 0}$ with $T_0 = T$, $T_1= \Ttilde$, $T_i \in \la_i\text{-}\mod$ for $i$ even and $T_i \in \mod\text{-}\la_i$ for $i$ odd, such that $T_{i +1} = \End_{\la_i}(T_i)^{\op}$ for $i$ even, and $T_{i +1} = \End_{\la_i}(T_i)$ for $i$ odd.

\begin{teorA*}
{\rm (For a more complete version, see Theorem \ref{teor.strong-tilting-iteration} and Corollary \ref{cor.dualities-by-strongtilting}.)} Suppose that $\la$ is a basic Artin algebra such that $\pinf(\lamod)$ is contravariantly finite in $\lamod$.  Let $_\la T$ be the corresponding basic strong tilting module and $\latilde = \End_\la(T)^{\op}$. 

If $\pinf(\modlatilde)$ is contravariantly finite in $\modlatilde$, the category $\lamod$ allows for unlimited iteration of strong tilting.  The resulting sequence of basic Artin algebras $\la = \la_0$, $\la_1$, $\la_2$, $\la_3,  \dots$ is ultimately periodic with period $2$:
$\la_i \cong \la_j$ whenever $i$ and $j$ are positive integers with the same parity. Their categories of modules of finite projective dimension are  linked by dualities
\vskip.05truein 
\centerline{$\pinf(\modla _1) \leftrightarrow \pinf(\Lambda _2\text{\rm-mod}) \leftrightarrow \pinf(\modla _3) \leftrightarrow \cdots$} 
\vskip.05truein
\noindent induced by the contravariant $\Hom$-functors of the corresponding basic strong tilting bimodules $_{\la_i} (T_i) _{\la_{i+1}}$, resp. $_{\la_{i+1}} (T_i) _{\la_{i}}$. 

In general, $\la_0$ need not be isomorphic to $\la_2$ however.
\end{teorA*}

% leading to a sequence of basic Artin algebras $\Lambda = \la_0, \la_1, \la_2, \la_3, \dots$ with $\la_i \cong \la_j$ for positive integers $i, j$ with the same parity; their categories of modules of finite projective dimension are 
%additionally linked by dualities $\pinf(\modla _1) \leftrightarrow \pinf(\Lambda _2\text{-mod}) \leftrightarrow \pinf(\modla _3) \leftrightarrow \cdots$ induced by the contravariant $\Hom$-functors of the corresponding basic strong {\it tilting bimodules\/} (Theorem \ref{teor.strong-tilting-iteration} and Corollary \ref{cor.dualities-by-strongtilting}).

To return to the problem of determining the contravariant finiteness status of
 $\pinf (\lamod)$ in the first place, we suppose that $e_1, \dots, e_n$ form a complete set of orthogonal primitive idempotents of $\la$ and that the subsum $e = e_1+ \dots + e_m$ includes all idempotents $e_i$ that give rise to simple left modules of infinite projective dimension.  Our second and third main results refer to this setting.
 
\begin{teorB*}
{\rm (For a complete version, see Theorem \ref{thm.reduction-to-simples-of-infteprojdim}.)}  Adopt  the notation of the preceding paragraph, and suppose that $e \Lambda (1 - e)$ has finite projective dimension as a left module over the corner algebra $e \Lambda e$.  Then $\pinf(\lamod)$ is contravariantly finite in $\lamod$ if and only if $\pinf(\elaemod)$ is contravariantly finite in $\elaemod$.
 \end{teorB*}
 
 Theorem \ref{thm.reduction-to-simples-of-infteprojdim} also spells out how, in case of existence,  the {\it minimal $($right$)$ $\pinf(\lamod)$-approximation\/} of $M \in \lamod$ relates to the minimal $\pinf(e \la e\text{-}\mod)$-approximation of $eM$.  The ensuing possibility of shucking off primitive idempotents in checking for contravariant finiteness of $\pinf (\lamod)$ not only provides effortless access to most of the cases in which this property has previously been established (formerly involving considerable effort), but yields contravariant  finiteness in far more general situations.

 Our argument relies on an exploration of the {\it torsion-torsionfree triple\/} $(\mathcal{C}, \mathcal{T}, \mathcal{F})$ which is associated to the idempotent $e$, and on the  {\it Giraud subcategory\/} $\mathcal{G}$ of $\laMod$ (in the sense of Gabriel \cite{Gabrielp}) corresponding to the hereditary torsion pair $(\mathcal{T}, \mathcal{F})$.

In order to extend Theorem B to an efficient test of whether the stronger conclusions of Theorem A hold for $\la$, we once more assume that $\la$ is a basic Artin algebra such that $\pinf(\lamod)$ is contravariantly finite.  As before, we let $e_1, \dots, e_n \in \la$ be a complete set of primitive idempotents, $T \in \lamod$  the basic strong tilting module, and $\latilde = \End_\la(T)^{\op}$. By $S_i$ we denote the simple left $\la$-module corresponding to $e_i$.  It is well-known that the indecomposable direct summands $T_1, \dots, T_n$ of $_\la T$ coincide, up to isomorphism, with the distinct indecomposable direct summands of $\bigoplus_{1 \le i \le n} \mathcal{A}_i$, where  $\mathcal{A}_i$ is the minimal $\pinf(\lamod)$-approximation of the indecomposable injective left module with socle $S_i$. Our analysis of the $T_i$ will pave the road towards showing that the question of unlimited iterability of strong tilting of $\lamod$ can in turn be played back to the corner algebra $e \la e$ for any idempotent $e$ as specified in Theorem B. 

\begin{teorC*}  Let $m \le n$ be chosen such that $\pdim_\la S_j < \infty$ for $j > m$.  Set $e = e_1 + \cdots + e_m$ and assume that $\pdim_{e \la e} e \la (1 - e) < \infty$.  Then:
 \smallskip
 
\noindent {\rm\textbf{(1)} (Proposition \ref{prop.decomposition-strongtilting}.)}  The number of distinct indecomposable direct summands of $\bigoplus_{1 \le i \le m} \mathcal{A}_i$ is $m$. Denote them by $T_1, \dots, T_m$ and set $T' = \bigoplus_{1 \le i \le m} T_i$.  

For $j \ge m+1$, the approximation $\mathcal{A}_j$ of $S_j$ decomposes in the form $\mathcal{A}_j =  T_j \oplus U_j$, where $T_j$ is indecomposable with the property that $S_j$ is the only simple module of finite projective dimension in $\soc T_j$, and all indecomposable direct summands of $U_j$ occur as direct summands of $T'$.
\smallskip

\noindent {\rm\textbf{(2)} (For more detail, see Theorem \ref{Birge5.3}, and Corollaries \ref{Birge5.4}, \ref{cor.iteration-stongtilting}.)}  $\lamod$ allows for unlimited iteration of strong tilting if and only if the same is true for $e \la e\text{-}\mod$.
\end{teorC*}
 
Our proof of Theorem C is based on connections between $T \in \lamod$ and the strong tilting objects in $e \la e$-mod and in the Giraud subcategory $\mathcal{G}$.

Earlier results for truncated path algebras (i.e., for algebras of the form 
$KQ/I$, where $K$ is a field, $Q$ a quiver, and $I$ the ideal generated by the paths of length $L+1$ for
 some $L \ge 1$), showcase the effects of iterated strong tilting and the associated dualities among the $\pinf$-categories encountered along the iterations \cite{Birgedualities}; these findings readily follow from the above theorems. In fact, the present results yield substantial generalizations of the picture that arose in the truncated case (Propositions \ref{prop.precyclic-contrfiniteness}, \ref{prop.precyclic-stiltingiteration}, Corollaries  \ref{cor.Birge-Theorem D} and \ref{cor.precyclic-stilting iteration}).  From our reduction technique it also follows that, for any left serial algebra
 $\Lambda$, both $\pinf (\lamod)$ and $\pinf(\modla)$ are contravariantly finite in $\lamod$ and $\modla$, respectively (the former fact had been known; see \cite{BH}).  Further applications address Artin algebras arising from Morita contexts, such as algebras of triangular matrix type (Theorem \ref{thm.mainMoritacontexts}, Corollary \ref{cor.triangular matrix}, Examples \ref{ejems.effective construction}), and the elimination of simple modules of low projective dimension in the test for contravariant finiteness ({Proposition \ref{prop.elimination-simples-projdim1}}).

Section 2 assembles minimal conceptual background and  builds the tools required for proving our main theorems.  Section 3 provides a general characterization of the situation in which $\lamod$ allows for unlimited iteration of strong tilting.  The results targeting tests for contravariant finiteness of $\lamod$ 
and for the availability of repeated strong tilts of $\lamod$ are contained in Sections 4 and 5, respectively.  In Section 6, we specify applications.

\section {Notation, background and auxiliaries}
Throughout, $\Lambda$ will be a basic Artin algebra, and $J$ its Jacobson radical. 
 We point out that the restriction to basic algebras does not affect the generality of our investigation; we adopt it 
for increased transparency of the underlying ideas. $\laMod$ and $\lamod$ 
stand for the categories of all (resp., 
all finitely generated) left $\la$-modules.   Further, $S_1, \dots, S_n$ will be isomorphism representatives of the 
simple objects in $\laMod$, and $\top M$, $\soc M$ 
will stand for the top and socle of $M \in \laMod$, 
respectively.  By 
$\mathcal{P}^{<\infty}(\laMod)$ (resp., $\mathcal{P}^{<\infty}(\lamod)$)
 we denote the full subcategory of 
$\laMod$ (resp., $\lamod$) 
having as objects the modules of finite projective dimension.  Moreover, for any finitely generated 
$\la
$-module $M$, we denote by $\add(M)$ the full subcategory of 
$\lamod$
 consisting of the direct summands of 
finite direct sums of copies of $M$.  The module $M$ is {\it basic\/} if it
 has no indecomposable direct summands of 
multiplicity $> 1$.

\subsection{Contravariant finiteness of  $\mathcal{P}^{<\infty}(\lamod)$ and strong tilting 
modules}
Following Miyashita \cite{Mia},
 we call a left $\Lambda$-module $T$ a {\it tilting module\/} in case {\bf (1)} $T$ belongs to 
$\mathcal{P}^{<\infty}(\lamod)$, {\bf (2)} $\Ext^i_{\Lambda}(T, T) = 0$ for $i \geq 1$, and {\bf (3)} there exists an exact sequence
 $0 \rightarrow  {_{\la} \la} \rightarrow T_0 \rightarrow \cdots \rightarrow T_m \rightarrow 0$ with $T_j \in \add(T)$.  It is well known that any basic tilting module has precisely
$n =$  rank $K_0(\Lambda)$
 indecomposable direct summands, and that any tilting module 
$T \in \lamod$ 
gives rise to a left-right symmetric situation as follows:  If 
$\latilde = \End_{\Lambda}(T)^{\op}$, 
then the right $\latilde$-module $T_{\latilde}$ is in turn a tilting module and $\End_{\latilde} (T)$ is canonically isomorphic to $\la$.  This justifies the reference to a {\it tilting bimodule\/} ${_{\Lambda} T_{\latilde}}$.  Strong tilting modules were first considered by Auslander and Reiten in \cite{AR1}.  We introduce them via a characterization equivalent to the original definition (see \cite{AR1}): 
 Namely, we call a tilting module $T$ {\it strong\/} if it satisfies the following relative injectivity condition in $
\mathcal{P}^{<\infty}(\lamod)$:  {\bf (4)} $\Ext^i( M ,T) = 
0$ for all $M \in \pinf(\lamod)$ and $i \geq 1$.  It was shown by Auslander 
and Reiten [loc.cit.] that $\lamod$ has a strong tilting module if and only if the category $\mathcal{P}^{<\infty}(\lamod)$ is {\it contravariantly finite in $\lamod$\/},  a property which we will recall next.  Moreover, according to [loc.cit], in case of existence, there is precisely one basic strong tilting module in 
$\lamod$, up to isomorphism.

The concept of contravariant finiteness of subcategories of $\lamod$ has its roots in 
work of Auslander and Smal\o\ \cite{AS1}: Namely, the
 category $\mathcal{P}^{<\infty}(\lamod)$ 
is {\it contravariantly finite in $\lamod$\/} provided that for each object $M \in \lamod$, 
the functor $\Hom_{\Lambda} ( -, M)_{|_{\mathcal{P}^{<\infty}(\la\text{-mod})}}$ is finitely generated in the category of additive contravariant functors $\mathcal{P}^{<\infty}(\lamod) \rightarrow \mathbf{Ab}$. 
 This condition translates into the following requirement for arbitrary $M \in \lamod$:  There is an object $A \in \mathcal{P}^{<\infty}(\lamod)$, together with a map $\phi \in \Hom_{\Lambda}(A, M)$, 
such that each map in $\Hom_{\Lambda}(\mathcal{P}^{<\infty}(\lamod), M)$ factors through $\phi$. 
 Any such  pair $(A, \phi)$ is called a ({\it right\/}) {\it $\mathcal{P}^{<\infty}(\lamod)$-approximation of $M$\/}.  Since we will only consider right approximations, we will frequently omit the qualifier ``right".  By a mild abuse of terminology, we will moreover refer to the domain $A$ of $\phi$ as a $\mathcal{P}^{<\infty}(\lamod)$-approximation of $M$.  Suppose $M$ has a 
$\mathcal{P}^{<\infty}(\lamod)$-approximation, say $\phi: A \rightarrow M$.  As was shown by Auslander-Smal\o\ in [loc.cit.], up to isomorphism, there is only one approximation
 $\phi: \mathcal{A}(M) \rightarrow M$ such that  $\mathcal{A} (M)$ has minimal length. 
 It is alternatively characterized by the condition that any endomorphism $u$ of $\mathcal{A} (M)$ which 
satisfies $\phi \circ u = \phi$ is an automorphism; we say that the map $\phi$ is {\it $($right$)$ minimal\/} if this implication holds.  Since it is unlikely to cause misunderstandings, we will also refer to $\mathcal{A}(M)$ as ``the" minimal $\mathcal{P}^{<\infty}(\lamod)$-approximation of $M$ whenever convenient.

The mentioned existence result by Auslander-Reiten will be crucial in the sequel.  We state it for easy reference.

\begin{teor} {\bf (1)} \label{teor.initial} {\rm \cite{AR1}} There exists a strong tilting module $T \in \lamod$ if and only 
if $\mathcal{P}^{<\infty}(\lamod)$ is contravariantly finite in $\lamod$. 
 In the positive case, the basic strong tilting module is unique up to isomorphism:  it is the direct sum of the distinct indecomposable modules $C \in \mathcal{P}^{<\infty}(\lamod)$ which satisfy 
$\Ext^i_{\Lambda}(\mathcal{P}^{<\infty}(\lamod), C) = 0$ for $i \geq 1$.
\smallskip

{\bf (2)} {\rm \cite [Supplement II in Section 2.A]{Birgedualities}}  A more explicit description of the indecomposable direct summands of a strong tilting module $T${, when it exists,} can be obtained from the following:  Let $\mathcal{A}$ be the minimal 
$\pinf (\lamod) $-approximation of the minimal injective cogenerator of
 $\laMod$.  Then  $\add(T) = \add(\mathcal{A})$.

\end{teor}

\subsection{The TTF-triple associated to an idempotent $e$ in $\Lambda$ } \label{subsec.TTFtriple}

We fix an idempotent element $e \in \Lambda$. 
 By Jans \cite {JANS} (see also 
\cite [Section VI.8] {STEN}), $e$ defines a TTF triple 
 $(\mathcal{C}_e,\mathcal{T}_e,\mathcal{F}_e)$  in the category $\laMod$; this means that the pairs 
$(\mathcal{C}_e,\mathcal{T}_e)$  and  $(\mathcal{T}_e,\mathcal{F}_e)$
are 
both torsion pairs.  The torsion, resp. torsionfree, classes are as follows:

(a) $\mathcal{C}_e$ consists of the  $\Lambda$-modules $C$ generated by {$\Lambda e$}, i.e., the modules of the form 

\ \ \ \, $C=\Lambda eC$.

(b) $\mathcal{T}_e$ consists of the left $\Lambda$-modules annihilated by $e$. 

(c) $\mathcal{F}_e$ consists of the $\Lambda$-modules $F$
with the property that the annihilator 

\ \ \ \, $\ann_{F}( e \Lambda)$ of $e \Lambda$ in $F$ is zero.

 \noindent Observe that the torsion pair $(\mathcal{T}_e,\mathcal{F}_e)$ is hereditary, whence the corresponding torsion 
radical is left exact (see \cite[Proposition VI.3.1]{STEN}).  On the other hand, 
$(\mathcal{C}_e,\mathcal{T}_e)$ fails to be hereditary 
in general.

%%%%%%%%%%%%%%%%%%%%%%%%%%%%%%%%%%%%%%%%%

{\bf Further notation:}  Since we will keep the idempotent $e$ fixed, we will more briefly write 
\vskip.06truein
\centerline{$(\mathcal{C}, \mathcal{T}, \mathcal{F})$\ \  for\ \  $(\mathcal{C}_e, \mathcal{T}_e, \mathcal{F}_e)$.}
\vskip.06truein
%%%%%%%%%%%%%%%%%%%%%%%%%%%%%%%%%%%%

\noindent The torsion radicals $\nabla$ and $\Delta$ associated to the pairs 
$(\mathcal{C} ,\mathcal{T})$  and $(\mathcal{T},\mathcal{F})$ are the 
 idempotent subfunctors of the identity functor on $\laMod$ given by:
\vskip.06truein
%%%%%%%%%%%%%%%%%%%%%%%%%%%%%%%%%%%%%%%%%%

\centerline{$\nabla(M) = \Lambda e M \ \ \ \text{and} \ \ \ \Delta(M) = \ann_{M} (e\Lambda) = \{ x \in M \mid e
 \Lambda x = 0\},$} 
\vskip.06truein

\noindent respectively. Note that $\nabla(M)$ is the largest submodule of $M$ with the property that all simple modules in the top of $\nabla(M)$ 
belong to $\mathcal{F}$ $\bigl($equivalently, $\top \bigl(\nabla(M)\bigr)$ belongs to ${\add} (\la e/ Je) \bigr)$, whereas $\Delta(M)$ is the largest submodule of $M$ which is annihilated by $e$ $\bigl($equivalently, all simple subfactors of $\Delta(M)$ belong to ${\add} \bigl(\la(1 -  e)/ J(1- e)\bigr)\bigr)$. 

A third functor $\laMod  \rightarrow \laMod $ will serve to render the constructions in Sections 4 and 5 more transparent.  It assigns to each $\la$-module $M$ the following subfactor of $M$: 
\vskip.06truein
\centerline{$\mathfrak{core}(M)\  =\  \nabla(M)/ \Delta\bigl(\nabla(M)\bigr)\ = 
\ \Lambda e M/ \Delta(\Lambda eM)$.}
\vskip.06truein
%%%%%%%%%%%%%%%%%%%%%%%%%%%%%%%%%%%%%%%%%

\noindent If $M$ is finitely generated, the core of $M$ has maximal length among the subfactors $V/ U$ of $M$ such that the top and socle of $V/U$ belong to $\mathcal{F}$, i.e., such that  
$\top(V / U), \soc(V/U) \in \add(\Lambda e / J e)$.  In fact, it can easily be seen that this maximality property determines 
$\mathfrak{core}(M)$ up to isomorphism.  Observe moreover, that $\mathcal{C} \cap \mathcal{F}$ consists of the $\la$-modules 
which coincide with their cores.

%%%%%%%%%%%%%%%%%%%%%%%%%%%%%%%%%%
\begin{ejem}  \label{ex.delta.nabla}  To illustrate these endofunctors, we 
let $\la =KQ/ \langle R \rangle$, where $K$ is a field, $Q$ the quiver 
 $$\xymatrixrowsep{3pc}\xymatrixcolsep{4pc}
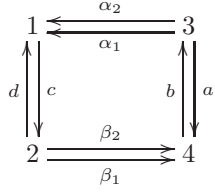
\xymatrix{1 \ar@<0.5ex> [d]^{c}  & 3   \ar@<0.5ex>  [l]^{\alpha_1} \ar@<- 0.5ex>  [l] _{\alpha_2}
 \ar@<0.5 ex> [d]^{a}  \\
2  \ar@<0.5ex>  [r]^{\beta_2} \ar@<-0.5ex>  [r] _{\beta_1}  \ar@<0.5ex>  [u]^{d}       &4  \ar@<0.5ex>[u]^{b}
}$$
\noindent and $R=\{ab, cd, dc, \beta_1c,\beta_2c,\}\cup\{\alpha_kb\beta_l\text{: }k,l\in\{1,2\}\}$.

We consider the TTF triple $(\mathcal{C},\mathcal{T},\mathcal{F})$ associated to  $e=e_1+e_2$  and apply the functors $\nabla$, $\Delta$,  $\mathfrak{core}$ to the indecomposable injective left $\Lambda$-modules shown below:
$$\xymatrixrowsep{1.5pc}\xymatrixcolsep{1pc}
\xymatrix{
 &  & && 3  \edge[d] &  & 3  \edge[d] && && &&&&\\
 3 \edge[d] & 3 \edge[d] &  && 4 \edge[d] &   & 4 \edge[d]&&  & &  && & & \\
4 \edge[d]  & 4 \edge[d]   & && 3  \edge[dr] &   &  3  \edge[dl] && 3 \edge[dr] &2 \edge[d] &2 \edge[dl]  
&& & & \\
 3 \edge[dr] & 3 \edge[d] & 2 \edge[dl]  &&   & 1  \edge[d] &     &&    & 4 \edge[d] &    &&   3 \edge[dr] &2 \edge[d] &2 \edge[dl]  \\
 & 1 & &&   & 2 &  && &3& && &4&
}$$
\medskip

\begin{enumerate}
\item $\Delta (I_j)=0$ $($i.e. $I_j\in\mathcal{F}$$)$ for $j=1,2$. The submodules  $\Delta (I_3)$ of $I_3$ and $\Delta (I_4)$ of $I_4$ are depicted by the following diagrams:
$$\xymatrixrowsep{1.5pc}\xymatrixcolsep{1pc}
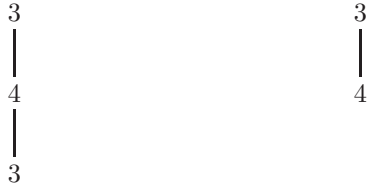
\xymatrix{
& && 3 \edge[d] &&&&&&& 3 \edge[d]\\
& && 4\edge[d] &&&&&&& 4\\
& && 3 &&&&&&& \\
} $$

\item The submodules $\nabla (I_j)$, for $j=1,...,4$, are depicted by the diagrams
$$\xymatrixrowsep{1.5pc}\xymatrixcolsep{1pc}
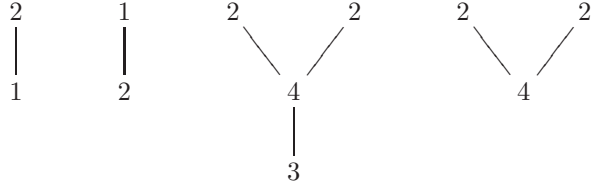
\xymatrix{
&& &  2 \edge[d]  &&  1\edge[d]   &&  2 \edge[dr] & & 2 \edge[dl] && 2 \edge[dr] & & 2\edge[dl] \\  
& && 1 && 2 &&               & 4 \edge[d]  &                                    &&   & 4 & \\
& &&  &&  &&               & 3 &                                    &&   &  & \\
} $$

\item Finally, $\mathfrak{core}(I_j)=\nabla (I_j)$, for $i=1,2$, and
$\mathfrak{core}(I_3)\cong S_2\oplus S_2\cong \mathfrak{core}(I_4)$. 
\end{enumerate}

\end{ejem}

 We point out that, in general, the core of a module may also have simple composition factors annihilated by $e$; for a broader spectrum of examples, see \cite{Birgedualities}, \cite{HuisgenZimmermann-Saorin}.

%%%%%%%%%%%%%%%%%%%%%%%%%

\subsection {The adjoint pair $(\Lambda  e \otimes_{e \Lambda  e} -,\, \bf {e})$ } \label{2.3}
\medskip

By $\bf {e}$ we denote the functor $\laMod \rightarrow e \Lambda e\text{-Mod}$ 
which sends $M$ to $eM$.
The two functors of the title thus play the role of induction and restriction in the exchange of information between $e \Lambda e$-modules on one hand and $\la$-modules on the other.
Accordingly, $\Lambda e \otimes_{e \Lambda e} - : e \Lambda e\text{-Mod} \rightarrow \laMod$ is left adjoint to 
$ \bf {e}$.

The unit corresponding to this adjunction is the obvious natural transformation
$$\eta:  1_{e \Lambda e\text{-Mod}}\  \longrightarrow\ { \bf {e}} \circ (\Lambda e \otimes_{e \Lambda e} -), \quad U \longmapsto e  (\Lambda e \otimes_{e \Lambda e} U).$$ 
Clearly, $\eta$ is an isomorphism of functors, whence $\Lambda e \otimes_{e \Lambda e} -$ is fully faithful (see \cite[{Proposition II.7.5}]{HS}).   

The counit $\epsilon: (\Lambda e \otimes_{e \Lambda e} - ) \circ {\bf {e}}\  \rightarrow\  1_{\laMod}$ of the adjunction
is given by the family $(\epsilon_M)_{M \in \laMod}$, where  $\epsilon_M: \Lambda e \otimes_{e \Lambda e} eM \rightarrow M$ sends $a \otimes x$ to $ax$.  Note that the image $\Im(\epsilon_M)$ equals $\nabla(M)$.

We briefly explore the functor ``restriction followed by induction", 
$\laMod \rightarrow \laMod$, which sends $M \in \laMod$ to 
$M^{\ddagger} : = \Lambda e \otimes_{e \Lambda e} eM$.  Evidently, this functor sees only the core of a $\la$-module $M$; indeed, $M^{\ddagger}$ is naturally isomorphic to $\mathfrak{core}(M)^{\ddagger}$.  On the other hand, it preserves this core, as the following lemma shows.

\begin{lema} \label{core} The functor $\mathfrak{core}:  \laMod \rightarrow  \laMod$ is naturally equivalent to the functor 
$$\bigl(\Lambda e \otimes_{e \Lambda e} e(-)\bigr)/ \Delta\bigl(\Lambda e \otimes_{e \Lambda e} e(-)\bigr).$$
More precisely, the counit $\epsilon$ of the adjunction induces a family of isomorphisms
$$ \overline{\epsilon}_M :  \mathfrak{core}(M^\ddagger) = (\Lambda e \otimes_{e \Lambda e} eM)/ \Delta(\Lambda e \otimes_{e \Lambda e} eM) \ \ \cong \ \   \mathfrak{core}(M), \ \ \  M \in \laMod,$$
which is natural in $M$.
\end{lema}
\begin{proof}  Clearly, $\epsilon$ gives rise to a family of epimorphisms  
$${\rho_M: \Lambda e \otimes_{e \Lambda e} eM \rightarrow \nabla(M) / \Delta(\nabla(M)) = \mathfrak {core}(M),}$$
and 
$\Ker(\rho_M)$ contains 
$\Delta(\Lambda e \otimes_{e \Lambda e} eM)$.  That, conversely, 
$\Ker(\rho_M)$ is contained in $\Delta(\Lambda e \otimes_{e \Lambda e} eM)$ follows from the facts that 
${\bf{e}}(\rho_M): e \Lambda e \otimes_{e \Lambda e} eM \rightarrow e \bigl(\mathfrak{core}(M)\bigr) = eM$ is 
an isomorphism in $\elaemod$ and $\Delta(\Lambda e \otimes_{e \Lambda e} eM)$ is the
 largest submodule of $\Lambda e \otimes_{e \Lambda e} eM$ which is annihilated by $e$.  We conclude that the
  maps $\overline{\epsilon}_M$ are indeed isomorphisms.  
\end{proof}
  
\subsection{The Giraud subcategory of $\laMod$ corresponding to the torsion pair $(\mathcal{T}, \mathcal{F})$} \label{Gir}

We apply well-known facts about localization with respect to a hereditary torsion class $\mathcal{T}$ 
to the specialized situation where $\mathcal{T} = \mathcal{T}_e $.  We refer the reader
 to \cite{Gentle} and  \cite [Chapter IX]{STEN} for detail. 

Recall that the {\it Giraud subcategory\/} $\mathcal{G}$ of $\laMod$ relative to the hereditary torsion pair $(\mathcal{T}, \mathcal{F})$
 is a realization, inside $\laMod$, of the quotient
 category $\laMod/ \mathcal{T}$;  here we identify the torsion class $\mathcal{T}$ with  the (full) localizing subcategory of $\laMod$ that has object class $\mathcal{T}$.  In particular, $\mathcal{G}$ is a Grothendieck category.  Concretely, $\mathcal{G}$ is the full subcategory of $\laMod$ whose objects are the torsionfree $\la$-modules $F$ with the following restricted injectivity property: $\Ext_{\Lambda}^1(X, F) = 0$ for all (cyclic) torsion modules $X \in \mathcal{T}$.

It is well known that the fully faithful inclusion functor $\iota: \mathcal{G} \rightarrow \laMod$ has an exact left adjoint $\sigma: \laMod \rightarrow \mathcal{G}$, which is identifiable with the quotient functor 
$\laMod \rightarrow \laMod/ \mathcal{T}$; in
 particular, the pair $(\mathcal{G}, \sigma)$ has the universal property of such a quotient.  We will write $M_\sigma$ for $\sigma(M)$.  In parallel, $f_\sigma = \sigma(f)$ is the map $M_\sigma \rightarrow N_\sigma$ induced by $f \in \Hom_{\Lambda}(M,N)$.

To describe $M_{\sigma}$ up to isomorphism, we abbreviate $M/\Delta(M)$ by $\overline{M}$ and let $E(\overline{M})$ be an injective envelope of $\overline{M}$.  Then $M_{\sigma}$ is the preimage under the canonical map $E(\overline{M}) \rightarrow E(\overline{M}) / \overline M$ of the torsion submodule $\Delta\bigl(E(\overline{M}) / \overline M \bigr)$ of the quotient.  In particular, $\sigma$ is the identity on the objects of $\mathcal{G}$.

The explicit description of $\sigma: \laMod \rightarrow \laMod$ reveals that this functor, in turn, sees only the core of a $\la$-module M, i.e., $M_\sigma = \mathfrak{core}(M)_\sigma$, and that it preserves cores, meaning that $\mathfrak{core}(M_\sigma)$ is canonically isomorphic to $\mathfrak{core}(M)$; in fact, $M_\sigma$ is an essential extension of $\mathfrak{core} (M)$ which is maximal relative to the requirement that this core be preserved. The following alternative incarnations of the category $\mathcal{G}$ will be helpful in Sections 4 and 5.   

\begin{lema} \label{2.2} Suppose $e = e_1 + \cdots + e_m$ is a decomposition of $e$ into primitive idempotents.
\smallskip

 {\rm 1.} {\rm \cite [Proposition XI.8.6]{STEN}} The categories $\mathcal{G}$ and 
$\elaeMod$ are equivalent.  Quasi-inverse 
equivalences send $F \in \mathcal{G}$ to $e F$ in one direction, and send $X \in \elaeMod$ to $(\Lambda e \otimes_{e \Lambda e} X)_\sigma$ in the other.

The indecomposable projective objects of $\mathcal{G}$ are $(\Lambda e_i)_\sigma$ for $1 \le i \le m$, and the indecomposable injectives are $\bigl(E(S_i)\bigr)_\sigma = E(S_i)$ for $1 \le i \le m$.

\smallskip

\smallskip {\rm 2.} {\rm [9]}  The subcategory $\mathcal{C} \cap \mathcal {F}$ of
 $\laMod$ 
is in turn equivalent to $\mathcal{G}$.  Quasi-inverse equivalences send 
$M = \mathfrak{core}(M)$ to $M_\sigma$; in reverse, $F \in \mathcal{G}$ is sent to $\mathfrak{core}(F) = \nabla(F)$.  In particular, the functors $\sigma$ and $\sigma \circ \mathfrak{core}$ from 
$\laMod$ to $\laMod$ 
are naturally isomorphic. 

The indecomposable projective objects of $\mathcal{C} \cap \mathcal{F}$
 are  $\mathfrak{core}(\Lambda e_i ) = \Lambda e_i / \Delta( \Lambda e_i)$ for $1 \le i \le m$, and the indecomposable injectives are
$\mathfrak{core}\bigl(E(S_i)\bigr)) = \nabla\bigl(E(S_i) \bigr)$ for $1 \le i \le m$.
\end{lema}

We add some notation:  The unit of the adjunction $(\sigma, \iota)$ is the  natural transformation 
$$\mu = (\mu_M): 1_{\laMod}\longrightarrow \iota\circ \sigma, \ \  \text{with}\ \mu_M \in \Hom_{\Lambda}( M, M_\sigma)\ \text{canonical}.$$ 
If $M$ is torsionfree, we identify $\mu_M$ with the inclusion map
 $M \hookrightarrow M_\sigma \subseteq E(M)$. 
 Clearly, $\mu_M$ is an isomorphism precisely when $M$ belongs to $\mathcal{G}$.
\smallskip

\begin{rem} \label{rem.everything restricts}
We point out that the mentioned functors linking the subcategories $\mathcal{A}$ of
 $\laMod$ $($resp.~of $\elaeMod$$)$ introduced in subsections 2.2-2.4 restrict to functors connecting the pertinent  intersections $\mathcal{A} \cap \lamod$, resp., $\mathcal{A} \cap \elaemod$ and 
$\mathcal{A} \cap \mathcal{G}$.
\end{rem}

\begin{ejem} \label{ex.first return to Ex 2.2} {\rm(Return to Example \ref{ex.delta.nabla}.)}
 In this instance, $(I_j)_\sigma =I_j$ for $j=1,2$, because $I_1$ and $I_2$ are injective objects of $\mathcal{F}$, and consequently of $\mathcal{G}$.  Now let $j \in \{3, 4\}$. In either case, $I_j/ \Delta(I_j) \cong S_2 \oplus S_2$, whence $E_j: = E( I_j/ \Delta(I_j)) \cong I_2 \oplus I_2$.  Since the torsion submodule of $(I_2 \oplus I_2) / (S_2 \oplus S_2)$ is zero, we obtain $(I_j)_\sigma = S_2  \oplus S_2$, and the map $\mu_{I_j}:I_j\longrightarrow (I_j)_\sigma=S_2\oplus S_2$ is the obvious projection.
 \end{ejem}

\subsection{The poset of essential $\Delta$-extensions of a morphism}
%\medskip

In constructing $\mathcal{P}^{<\infty}(\lamod)$-approximations of $\Lambda$-modules $M$
 from $\mathcal{P}^{<\infty}(e\Lambda e\text{-mod})$-approximations of $eM$, 
 passage to maximal extensions of the type described in this subsection will be crucial.  Throughout we refer to the torsion theory $(\mathcal{T}, \mathcal{F})$ introduced in 2.2.
 
\begin{defi*}
{\rm
Given a morphism $f:M\longrightarrow Y$ in {$\lamod$},
 we consider the following \emph{eligible extensions of $f$\/}:  These are the pairs $(L,g)$, where $L$ is an essential extension of $M$ with the additional property that $L/M\in\mathcal{T}$ and $g \in \Hom_ \Lambda (L, Y)$
 satisfies $g_{| M}=f$.  

We say that two eligible pairs $(L,g)$ and $(L',g')$ are \emph{isomorphic} if there is an isomorphism $\psi : L\rightarrow L'$ such that $g' \circ \psi = g$.  An \emph{essential $\Delta$-extension of $f$\/} is the isomorphism class $[(L,g)]$ of an eligible pair $(L,g)$.  

The set $\mathcal{E}_f$ of all essential $\Delta$-extensions $[(L,g)]$ of $f$ is a poset under the following partial order: $[(L,g)]\preceq [(L',g')]$ in case there is a monomorphism $\psi :L\longrightarrow L'$ such that $g' \circ \psi  =g$.
}
\end{defi*}

We comment on the legitimacy of the final definition: Welldefinedness of the relation $\preceq$ is clear.  To check that it is antisymmetric, it suffices to observe that the existence of monomorphisms $L\longrightarrow L'$ and $L' \longrightarrow L$ forces the finitely generated $\Lambda$-modules $L$ and $L'$ to have the same length, whence monomorphisms between them are isomorphisms. 

\begin{prop} \label{prop.poset of Delta-extensions}Suppose that $M$ and $Y$ are  finitely generated 
torsionfree left $\Lambda$-modules, and let $f \in \Hom_\Lambda(M, Y)$.  Then the poset $(\mathcal{E}_f,\preceq )$ has a maximum.  \\ 
\smallskip
\noindent $\bullet$ In case $Y \in \mathcal{G}$, this maximum is $[(M_\sigma ,f_\sigma{:M_\sigma\longrightarrow Y_\sigma =Y})]$.% where $f_\sigma :M_\sigma\rightarrow Y_\sigma =Y$ is the image of $f$ under the functor $\sigma$.  
\\
\smallskip
\noindent  $\bullet$  For a general torsionfree module $Y$, the maximum of $\mathcal{E}_f$ is $[(N,\phi)]$, where $\phi: N\longrightarrow Y$ results from the pullback of $(f_\sigma, \mu_Y)$:
$$\xymatrixrowsep{3pc}\xymatrixcolsep{4pc}\xymatrix{N \ar[r]^{\phi} \ar[d]_{\chi}  &Y \ar[d]^{\mu_Y}  \\
M_\sigma \ar[r]^{f_\sigma}  &Y_\sigma
}$$
\end{prop}

\begin{proof}
In light of the hypothesis $\Delta(M) = 0$, we have $ M \subseteq M_\sigma \subseteq E(M)$ for an injective envelope $E(M)$ of $M$, where the first inclusion coincides with the map $\mu_M: M \hookrightarrow M_\sigma$ 
of 2.3.  Due to the equality $M_\sigma / M = \Delta(E(M)/ M)$, we find $M_\sigma$ to be the largest among the submodules $M'$ of $E(M)$ which contain $M$ and satisfy $M'/M \in \mathcal{T}$.   Consequently, all elements of
 $\mathcal{E}_f$ are represented by pairs $(L,g)$ with $M \subseteq L \subseteq M_\sigma$, and 
we may restrict our focus to such representatives.  Conversely, whenever $h:  N \rightarrow Y$ is a homomorphism with $M \subseteq N \subseteq M_\sigma$ and $h_{|M} = f$, the pair $(N, h)$ represents a class in $\mathcal{E}_f$.

To address our first claim, we note that the lengths of increasing sequences in $\mathcal{E}_f$ are bounded from above by the composition length of $E(M)$.  In particular, each element of $\mathcal{E}_f$ is majorized by a maximal element in this set. Therefore, we only need to prove that any two elements $[(L,g)]$ and $[(N,h)]$ of
$\mathcal{E}_f$ have a common upper bound.  Consider the restrictions $g_{| X},h_{| X}:X\longrightarrow Y$ to the intersection $X = L \cap N$.  Since the difference $g_{|X}-h_{|X}$ vanishes on $M$, it induces a homomorphism $X/M\longrightarrow Y$, which in turn vanishes because $X/M\in\mathcal{T}$ and $Y\in\mathcal{F}$.  The resulting equality $g_{|X}=h_{|X}$ guarantees that the assignment $\rho: L+N\longrightarrow Y$ with $\rho (l+m)=g(l)+h(m)$ is a well-defined map in $\Hom_\Lambda(L+N, Y)$.  In light of $L + N \subseteq M_\sigma$, the pair
  $(L+N,\rho )$ represents an element in $\mathcal{E}_f$ which majorizes both 
$[(L,g)]$ and $[(N,h)]$.   

Now suppose that $Y \in \mathcal{G}$.  Then $f_\sigma:  M_\sigma \rightarrow Y_\sigma \cong Y$ gives rise to an essential $\Delta$-extension $[(M_\sigma ,f_\sigma )]$.  That this extension is maximal in $\mathcal{E}_f$ is immediate from the description of $M_\sigma$. 

To prove the final claim, we note that the map $\chi$ in the pullback diagram is an injection since $\mu_Y$ is.  Thus the submodule $ \mu_M^{-1} (\chi(N))$ of $M_\sigma$ is isomorphic to $N$.  On replacing $N$ by this copy and adjusting $\phi$ and $\chi$ accordingly, we obtain another pullback diagram for $(f_\sigma, \mu_Y) $ with $M \subseteq N \subseteq M_\sigma$, where $\chi$ is now replaced by the inclusion map. This shows that $(N, \phi)$ is an eligible pair in the sense of the above definition.
 To check maximality of $[(N, \phi)]$ in $\mathcal{E}_f$, suppose that this $\Delta$-extension is majorized by   $[(N', \phi')]$, where $M \subseteq N'  \subseteq M_\sigma$; let $\psi: N \hookrightarrow N'$ be an embedding with 
$\phi = \phi' \circ \psi$.  Since $\psi$  extends to an automorphism of 
$E(M)$, we do not lose generality in viewing $\psi$ as a set inclusion.   
We thus obtain a commutative diagram, in which $\kappa$ denotes the inclusion map $N' \hookrightarrow M_\sigma$: 
$$\xymatrixrowsep{4pc}\xymatrixcolsep{6pc}
\xymatrix{N' \ar@/^2pc/[rr]^{\phi'} \ar@{-->}@/_1pc/[r]_(0.6){v} \ar@{_(->}[dr]_{\kappa}  &N \ar@{_(->}[l]_(0.4){\psi} \ar[r]^(0.4){\phi} \ar@{^(->}[d]^{\chi}  &Y \ar@{^(->}[d]^{\mu_Y}  \\&M_\sigma \ar[r]_{f_\sigma}  &Y_\sigma}$$
\noindent  Then $f_\sigma \circ \kappa \circ \psi  = \mu_Y \circ \phi' \circ \psi$, whence the restrictions of $f_\sigma \circ \kappa$ and $\mu_Y \circ \phi'$ to the domain $N$ of $\psi$ coincide.  Since the epimorphic image $N'/N$ of $N'/M$ is a torsion module, while $Y_\sigma$ is torsionfree, the argumentation backing the first claim shows that $f_\sigma \circ \kappa = \mu_Y \circ \phi'$.  The universal property of the pullback therefore yields a map $v \in \Hom_\Lambda (N', N)$ with $\phi' = \phi \circ v$ 
and $\kappa = \chi \circ v$.  In particular, $v$ is a monomorphism.  We infer that $N$ and $N'$ have the same length and conclude that $\psi$ is an isomorphism.   This proves $[(N',\phi')] = [(N, \phi)]$ as required.
\end{proof}

\section {Iteration of strong tilting} \label{sec.stilting-iteration}

When can the process of strongly tilting $\lamod$ to a category of right modules, 
$\modlatilde$, be iterated?  In the positive case, how do the resulting sequences of strongly tilted module 
categories behave?  The main purpose of this section is to answer these questions.  More specifically, we will show 
that, whenever $\lamod$ can be strongly tilted twice, the process can be iterated arbitrarily and turns 
periodic after the initial step.  Roughly speaking: In case $\pinf(\modlatilde)$ is in turn contravariantly finite in 
$\modlatilde$, the initial transition from $\la$ to $\latilde$ increases the homological symmetry by increasing 
the number of simple modules of finite projective dimension. 
This symmetrization makes the subsequent sequence 
$\modlatilde  \rightsquigarrow  \latiltilmod  \rightsquigarrow  \cdots$
 periodic with period $2$.  

The statement of the following theorem is based on part (1) of Theorem \ref{teor.initial}.

\begin{teor}\label{teor.strong-tilting-iteration}
 Let $\la$ be a basic Artin algebra such that $\pinf(\lamod)$ is
 contravariantly finite in $\lamod$.  Moreover, let ${_{\Lambda }T}$ be 
the corresponding basic strong tilting module, and $\latilde = \End_{\Lambda}(T)^{\op}$. 
 Suppose that $\pinf(\modlatilde)$ is in turn contravariantly finite in $\modlatilde$,
 thus giving rise to a basic tilting bimodule ${_{\widetilde{\latilde}} \widetilde{T}}_{\latilde}$  which is strong in 
$\modlatilde$.  Then $\widetilde{T}$ is strong on both sides. 

In particular, the process of strongly tilting $\lamod$ then allows for unlimited iteration, yielding a sequence of basic Artin algebras $\la=  \la_0,  \la_1, \la_2, \la_3, \dots$ with the property that $\la_i$ and $\la_j$ are isomorphic whenever $i$ and $j$ are positive integers with the same parity.

Moreover, the algebras $\la_0$ and $\la_2$ are isomorphic precisely when the tilting bimodule
 $_{\la} T_{\latilde}$ is strong on both sides, i.e., when $T = \widetilde{T} $. 
\end{teor}

\begin{Com}
 If, in the hypotheses of Theorem \ref{teor.strong-tilting-iteration}, the qualifiers ``basic" are dropped, one 
obtains Morita equivalent pairs of algebras $(\la_i, \la_{i + 2})$ for all $i \ge 1$.
\end{Com}

\begin{proof}  By a result of Auslander and 
Green (see \cite[Proposition 6.5]{AR1} and \cite[Proposition 7.1]{D-HZ} for a short argument), a tilting
 bimodule $_A U_B$ over Artin algebras $A$ and $B$ is strong in $A\text{-mod}$ if and only if all simple right $B$-modules embed into $U_B$.  In light of the hypothesis that $_{\la} T$ is
 a strong tilting module, we thus find that all simple right $\latilde$-modules are contained in
 $\soc T_{\latilde}$, up to isomorphism.  
Moreover, we deduce that
 strongness of $_{\widetilde{\latilde}} \widetilde{T}$ as
 a tilting object in 
$\latiltilmod$ will follow if we can show that all simple
 right $\latilde$-modules embed into $\widetilde{T}_{\latilde}$.  

To realize such an embedding, let $(\latilde/ \Jtilde)_{\latilde} = \widetilde{S}_1 \oplus \cdots \oplus \widetilde{S}_n$, where the $\widetilde{S}_i$ are simple. 
 Further denote by $E(\widetilde{S}_1 \oplus \cdots \oplus \widetilde{S}_n)$ an
 injective envelope, and let $\widetilde{\mathcal{A}}$ be a minimal
 $\pinf(\modlatilde)$-approximation of the latter. 
 In light of our hypothesis that $\widetilde{T}_{\latilde}$ is a 
strong tilting object in $\modlatilde$, the 
categories $\add(\widetilde{T}_{\latilde})$ and $\add(\widetilde{\mathcal{A}})$
 coincide; see Theorem \ref{teor.initial}(2).  Embed $T_{\latilde}$ into an
 injective right $\latilde$-module, say 
$\iota:  T_{\latilde} \hookrightarrow \bigoplus_{1 \le i \le n} E(\widetilde{S}_i)^{m_i}$.
  For a suitable exponent $m$, we obtain a $\pinf(\modlatilde)$-approximation
 $\phi: \widetilde{\mathcal{A}}^m \rightarrow  \bigoplus_{1 \le i \le n} E(\widetilde{S}_i)^{m_i}$.  Since $T_{\latilde}$ belongs to $\pinf(\modlatilde)$,
 the embedding $\iota$ factors through $\phi$, say $\iota = \phi \circ f$ for some 
$f \in \Hom_{\latilde}(T, \widetilde{\mathcal{A}}^m)$. 
 Then $f$ is injective, and we deduce that all simple right $\latilde$-modules embed into $\widetilde{\mathcal{A}}^m$. 
 Clearly, this implies that all simples in $\modlatilde$ embed into $\soc \widetilde{\mathcal{A}}$. 
 Invoking the fact that $\add( \widetilde{\mathcal{A}}) = \add(\widetilde{T}_{\latilde})$, we conclude that all 
$\widetilde{S}_i$ occur in $\soc \widetilde{T}_{\latilde}$ as well, 
which proves two-sided strongness of ${_{\widetilde{\latilde}} \widetilde{T} }_{\latilde}$.
 
 The final assertion follows from the preceding argument.
%and the argument is complete. 
\end{proof}

The result is sharp in the sense that the existence of a strong tilting module in $\lamod$ 
does not, by itself, provide sufficient homological symmetry to allow for iteration of strong tilting.

\begin{ejem} \label{ejem.Igusa-Smalo-Todorov}
%\definition{Example}  \modlatilde
This example is due to Igusa-Smal\o-Todorov \cite{Igusa}.  Let $\la = KQ/I$, where $Q$ is the quiver
$$\xymatrixcolsep{5pc}
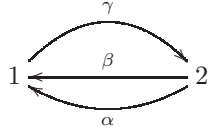
\xymatrix{
1 \ar@/^1.75pc/[r]^{\gamma}  &2 \ar[l]_{\beta} \ar@/^1pc/[l]^{\alpha}
}$$
\noindent and $I \subseteq KQ$ is the ideal generated by $\beta \gamma, \gamma \alpha, \gamma \beta$.  Clearly the right socle of $\la$ contains both simple right $\la$-modules, whence 
${\rm {l.findim }}(\la) = 0$
 (see \cite{Bass}).  
Hence the basic strong tilting object in 
$\lamod$ is $T ={ {}_\Lambda\Lambda}$, and
 $\latilde$ is isomorphic to $\la$.  However, the category 
$\pinf(\modlatilde) = \pinf(\modla)$ 
fails 
to be contravariantly finite in $\modla$ (see \cite{Igusa}). %Therefore we have that $n(\Lambda )=1$
%\enddefinition
 \end{ejem}

For a first application, we combine Theorem \ref{teor.strong-tilting-iteration} with the fact that strong tilting modules induce contravariant equivalences of categories of modules of finite projective dimension 
(see \cite [Reference Theorem III and Theorem 1]{Birgedualities} 
or \cite[Theorems 7,8]{HuisgenZimmermann-Saorin}).  This yields

\begin{cor} \label{cor.dualities-by-strongtilting}
Retain the notation and hypotheses of {\rm Theorem \ref{teor.strong-tilting-iteration}}.  
Then the pairs of functors $\bigl(\Hom_{\la}(-, T)$, $\Hom_{\latilde}(-, T)\bigr)$ and
 $\bigl(\Hom_{\latilde}( -,  \widetilde{T} ), \Hom_{\widetilde{\latilde}}( - , \widetilde{T})\bigr)$ induce dualities 
$$\pinf(\lamod)\  \longleftrightarrow\  \pinf(\modlatilde) \cap {^\perp(T_{\latilde}}) \ \ \ \subseteq  \ \ \ \pinf(\modlatilde)\  
\longleftrightarrow\  \pinf(\widetilde{\latilde}\text{-}\mod),$$
where ${^\perp(T_{\latilde})}$ is the full subcategory of $\modlatilde$
 whose objects are the modules $\widetilde{M}$ with 
$\Ext^i_{\latilde} (\widetilde{M}, T) = 0$ for $i \ge 1$.  Moreover, 
${^\perp(T_{\latilde})}$ contains $\pinf(\modlatilde)$ if and only if 
$_{\la} T _{\latilde}$ is two-sided strong, i.e., precisely when $T \cong \widetilde{T}$.
\end{cor}

In case the process of strongly tilting ${\lamod}$ permits for iteration, periodicity of the sequence 
${\lamod}  \rightsquigarrow \modlatilde \rightsquigarrow \cdots$ may set in with delay.   In fact, if $\la$ is 
a truncated path algebra one always obtains a sequence of basic strong tilts, 
${\lamod}  \rightsquigarrow 
\modlatilde  \rightsquigarrow \cdots$ (see \cite [Theorem 19]{Birgedualities});  
for this class of algebras, 
$\lamod \approx \widetilde{\latilde} \text{-mod}$ 
if and only if $Q$ does not have a precyclic source (see \cite [Corollary 7.2]{D-HZ}).   We refer to \cite{Birgedualities} and \cite{HuisgenZimmermann-Saorin}  for examples based on quivers with precyclic sources where the transition from the tilting module $T_{\latilde}$ to the basic strong tilting module $\widetilde{T}_{\latilde}$ in $\modlatilde$ is displayed.  

We conclude the section with a variant of the criterion of Auslander and Green for  a strong tilting module $T \in \lamod$ to be strong also as a tilting module over $\End_{\la}(T)^{\op}$; it was implicitly used in the proof of Theorem \ref{teor.strong-tilting-iteration}.  The upcoming version of the criterion does not rely on structural information regarding $_{\la} T$ and provides additional motivation for the detection of twosided strongness.

\begin{Obs} \label{OBS3.5} Suppose $\la$ is an Artin algebra such that 
$\pinf(\lamod)$ is contravariantly finite in $\lamod$.
  Then the following statements are equivalent:
\smallskip

\noindent {\bf (1)}  The basic strong tilting module $T \in \lamod$ is strong also in
 $\modlatilde$, where $\latilde = \End_{\la}(T)^{\op}$.
\smallskip

 \noindent {\bf (2)}  Every simple left $\la$-module of infinite projective dimension embeds into a $\la$-module of finite projective dimension.
 \smallskip
 
 \noindent{\bf(3)}  There exists a duality  
$\pinf (\lamod) \longleftrightarrow \pinf(\modlatilde)$.
\smallskip\vskip.03truein 
 
\noindent If {\rm (1)-(3)} are satisfied, then $\Hom_{\la}(-, T)$ and $\Hom_{\latilde}(-, T)$ induce  $(${unique up to isomorphism}$)$ quasi-inverse dualities 
 $\pinf (\lamod) \  \longleftrightarrow\  \pinf(\modlatilde)$.  
 \end{Obs}
 
 \begin{proof}
  ``(1)$\implies$(2)" is immediate from the criterion of Auslander and Green quoted in the proof of Theorem \ref{teor.strong-tilting-iteration}.  ``(2)$\implies$(1)":   By (2), every simple left $\la$-module $S$ embeds into a module $M = M(S)$ of finite projective dimension.  Consider an embedding $\iota:  M \hookrightarrow E$ of $M$ into an injective module $E$, and let  $\phi: \mathcal{A} \rightarrow E$  be a minimal $\pinf(\lamod)$-approximation.  Then $\iota$ factors through $\phi$, which shows $M$ to be contained in $\mathcal{A}$ up to isomorphism.  Hence so is $S$.  Since $\add(\mathcal{A}) \subseteq \add(T)$ by hypothesis (cf. the proof of Theorem \ref{teor.strong-tilting-iteration}), we conclude that $S$ embeds into $T$. Thus an application of \cite[Proposition 6.5]{AR1} yields (1).
 
The equivalence ``(1)$\iff$(3)" follows from  \cite[Reference Theorem III and Theorem 1]{Birgedualities} or \cite[Theorems 7,8]{HuisgenZimmermann-Saorin}.  For the supplementary statement, see \cite[Corollary 9]{HuisgenZimmermann-Saorin}.
 \end{proof}

\section{Testing for contravariant finiteness of the category  $ \mathcal{P}^{<\infty}(\lamod)$}

Our main objective in this section is to show that, in exploring whether the subcategory $\mathcal{P}^{<\infty}(\lamod)$  is contravariantly finite in $\lamod$, a comparatively mild hypothesis (\ref{condition4.1}(ii) below) permits us to eliminate primitive idempotents which correspond to simple left $\la$-modules of finite projective dimension.  This reduces the problem to a more manageable corner $e \Lambda e$ of $\la$.

\subsection{The setting}\label{condition4.1}

\begin{setting}[Blanket hypotheses] \label{set.initial setting} 
Throughout this section, we will assume that  the idempotent $e \in \Lambda$ satisfies the following two conditions:\\
 {\bf (i)} The semisimple left $\Lambda$-module $\Lambda (1-e)/J(1-e)$ has finite projective dimension.\\
 {\bf (ii)} $e\Lambda (1 - e)$ has finite projective dimension as a left $e\Lambda e$-module. 
\end{setting}

\begin{rems} \label{rems.about-Setting 4.1}
 $\bullet$ Suppose that $e = e_1 + \cdots + e_m$ and $1 - e = e_{m+1} + \cdots + e_n$, where the $e_i$ are orthogonal primitive idempotents of $\la$.  Then the conditions under \ref{condition4.1} amount to finiteness of the projective dimensions of the simple $\la$-modules $S_i = \Lambda e_i/ J e_i$ for $ i \geq m+1$ and those of the
 $e \Lambda e$-modules $e \Lambda e_i$ for  $i \geq m+1$.  In particular,   \ref{condition4.1}(i) implies that, 
for all $M \in \lamod$, 
\vskip0.05truein

\centerline{$\pdim_{\la} M < \infty \iff \pdim_{\la} \mathfrak{core}(M) < \infty \iff \pdim_{\la} M_\sigma < \infty $.}
\vskip0.05truein

\noindent With regard to \ref{condition4.1}(ii), we mention that, in case $\la$ is a path algebra modulo relations, the quiver and relations of $e \Lambda e$ are available from those of $\la$, making this condition computationally accessible. 
\smallskip

\noindent $\bullet$  Suppose that $S_1, \dots, S_m$ are precisely the simple left $\la$-modules of infinite projective dimension; in particular, this means that \ref{condition4.1}(i) is satisfied.  Even over monomial algebras, condition \ref{condition4.1}(ii) is not automatic in this situation, as is witnessed by an example of Fuller and Saor\'\i n \cite [ Example 4.2]{FS}.  However, if all of the supplementary simples $S_{m+1}, \dots, S_n$ have projective dimension at most $1$, then \ref{condition4.1}(ii) does always hold in this scenario [loc.cit.].    \end{rems}
 
 %Our next example shows, among other things,  that  Setting  \ref{set.initial setting} is not left-right symmetric for an idempotent $e\in\Lambda$.

In the present setting, the adjoint pair $(\Lambda e \otimes_{e \Lambda e} -,\ {\bf{e}})$ is particularly
useful towards the transfer of homological information between the categories $\lamod$ and $\elaemod$.  Namely:

\begin{lema} \label{lem.preserving-reflecting finprojdim}
%\proclaim{Lemma 3.2} %
Suppose $\Lambda$ satisfies the conditions under \ref{set.initial setting}.  Then the
 functors ${\bf{e}}: \lamod \rightarrow e \Lambda  e\text{-mod}$ and its left adjoint 
$\Lambda e \otimes_{e \Lambda e} - : \elaemod \rightarrow  \lamod$ preserve and reflect finite projective dimension. 
%\endproclaim
\end{lema}
\begin{proof}
By \cite[Lemma 1.2(b)]{FS}, the functor ${\bf{e}}:\lamod \rightarrow e\Lambda e\text{-mod}$ preserves finite projective dimension.  In light of the canonical isomorphism 
${\bf{e}} \circ(\Lambda e \otimes_{e \Lambda e} -) \cong \id_{\elaemod}$ 
this shows that $\Lambda e \otimes_{e \Lambda e} - $ reflects finite projective dimension.

To see that $\Lambda e \otimes_{e \Lambda e} - $ preserves finite projective dimension,  let $X \in \pinf(\elaemod)$, and let

\vskip0.05truein

\centerline{$\XX:\ \   0 \rightarrow\, Q_m\, {\overset{f_m} \longrightarrow}\, \cdots\, {\overset{f_2} \longrightarrow}\; Q_1\, {\overset{f_1} \longrightarrow}\, Q_0\,  {\overset{f_0} \longrightarrow}\,  X\, \rightarrow 0$}
\vskip0.05truein
\noindent be a 
projective resolution in $\elaemod$.  Then all terms $\Lambda e \otimes_{e \Lambda e} Q_j$ of the complex $\Lambda e 
\otimes_{e\Lambda e} \XX$ are projective over $\la$; indeed, the $Q_j$ belong to $\add(_{e \Lambda e} e \Lambda e)$, whence the 
modules $\Lambda e \otimes_{e \Lambda e} Q_j$ belong to the category $\add(\Lambda e \otimes_{e \Lambda e} e\Lambda e)= \add({_{\Lambda} \Lambda e})$.  
Moreover, all homology modules of $\Lambda e \otimes_{e\Lambda e} \XX$ have finite projective dimension in $ \lamod$; indeed, due to the natural isomorphism ${\bf{e}} (\Lambda e \otimes_{e\Lambda e} \XX)\, \cong\,  \XX$, they are all 
annihilated by ${\bf{e}} $, whence their simple composition factors are direct summands of $\la(1 - e)/ J(1 - e)$.   Writing $(F_i )$ for the differential $(\Lambda e \otimes_{e \Lambda e} f_i)$ of $\Lambda e \otimes_{e\Lambda e} \XX$, we
 find that $\Ker(F_m)$ has finite projective dimension.  In view of projectivity of $\Lambda e \otimes_{e\Lambda e} Q_m$, so does $\Im(F_m)$, whence we obtain finiteness of $\pdim_{\Lambda} \Ker(F_{m-1})$ from that of
 $\pdim_{\Lambda} \Ker(F_{m-1})/ \Im (F_m)$.  An obvious induction thus shows that $\Im(F_0) \cong \Lambda \otimes_{e \Lambda e} X$ has finite projective dimension over $\Lambda$. 

That the functor ${\bf{e}} $ reflects finite projective dimension is now seen as follows:  Given $M \in  \lamod$ such that ${\bf{e}} (M)$ has finite projective dimension over $e \Lambda e$, we apply the conclusion of the preceding paragraph and Lemma \ref{core} to deduce that $\pdim_{\Lambda} \mathfrak{core}(M) < \infty$.  Once more invoking  \ref{condition4.1}(i), we conclude that $\pdim_{\Lambda} M < \infty$.  
\end{proof}

The following consequences of Lemma \ref{lem.preserving-reflecting finprojdim}
 will facilitate applications of the main result of this section.

First we combine Lemma \ref{lem.preserving-reflecting finprojdim}
 with the equivalences $e \Lambda e\text{-Mod} \approx \mathcal{G} \approx {\mathcal{C}\cap\mathcal{F}}$ of 
Lemma \ref{2.2}, to find that finiteness of projective dimensions transfers smoothly among the categories singled out in Section \ref{Gir}.  Namely:  If $M_\sigma \in \laMod$ is an object of the Giraud subcategory $\mathcal{G}$ of $\laMod$, 
then $M_\sigma$ has finite projective dimension in $\mathcal{G}$ if and only if 
$\pdim_{e \Lambda e} e M_\sigma < \infty$,
 if and only if $\pdim_{\Lambda} M_\sigma < \infty$.  Analogously: If $M = \mathfrak{core}(M)$ is an object of 
 ${\mathcal{C}\cap\mathcal{F}}$, then $M$ has finite projective dimension in this subcategory if and only 
if $\pdim_{\Lambda} M  < \infty$.   This leads to

\begin{lema}\label{Lemma 4.4}
Let $M \in \lamod$, suppose that $eM$ has a $\pinf(\elaemod)$-approximation, 
and let  $q:  X \rightarrow eM$ be a minimal such approximation. 
 Then $(\Lambda e \otimes_{e \Lambda e} q)_\sigma: (\Lambda e \otimes_{e \Lambda e} X)_\sigma \rightarrow (\Lambda e \otimes_{e \Lambda e} eM)_\sigma \cong M_\sigma$ is a minimal $\pinf(\lamod)$-approximation of $M_\sigma$.
\end{lema}

\begin{proof}
 Referring to the explicit category equivalences of Lemma \ref{2.2}, we
 see that $(\Lambda e \otimes_{e \Lambda e} q)_\sigma$ is a minimal approximation 
of $M_\sigma$ with respect to the subcategory $\pinf\bigl(\mathcal{G} \cap(\lamod)\bigr)$ 
of $\mathcal{G} \cap(\lamod)$.  
In light of the preceding remarks, the domain of this map has finite projective dimension also in 
$\lamod$. Right minimality of $(\Lambda e \otimes_{e \Lambda e} q)_\sigma$  in $\lamod$
 is clear. To see that $(\Lambda e \otimes_{e \Lambda e} q)_\sigma$ is actually a 
$\pinf(\lamod)$-approximation 
of $M_\sigma$, let $U \in \pinf(\lamod)$ and $f \in \Hom_{\Lambda}(U, M_\sigma)$.  Since $M_\sigma$ is torsionfree and $U/\Delta(U)$ again belongs to $\pinf(\lamod)$, we do not 
lose generality in assuming that $U$ is torsionfree.  Consequently, $\mu_U: U \rightarrow U_\sigma$ is 
an embedding with the property that $U_\sigma /U \in \mathcal{T}$.  Now
 the restricted injectivity property of $M_\sigma$ yields an 
extension of $f$ to $f^*: U_\sigma \rightarrow M_\sigma$, i.e., $f = f^*\circ \mu_U$.  Since $U_\sigma$ has finite projective dimension in $\mathcal{G} \cap (\lamod)$, we further obtain a
 factorization $f^* =  (\Lambda e \otimes_{e \Lambda e} q)_\sigma \circ g$ for some 
$g \in \Hom_{\Lambda}(U_\sigma,  (\Lambda e \otimes_{e \Lambda e} X)_\sigma)$.  This clearly yields a factorization 
of $f$ through $(\Lambda e \otimes_{e \Lambda e} q)_\sigma$.   
\end{proof}

The next observation depends only on Condition \ref{condition4.1}(i). 

\begin{lema}\label{lem.P-cover of torsionfree}  
Let $M\in \lamod$, and let  $p:\mathcal{A}(M) \rightarrow M$ be a 
minimal $\pinf(\lamod)$-approximation. Suppose $p$ factors in the 
form $p = \tau \circ  \rho$ 
with $\rho \in \Hom_{\Lambda}(\mathcal{A}(M),N)$ and $\tau \in \Hom_{\Lambda}(N, M)$.
\smallskip

If $\pdim N <\infty$ and $\rho$ is an epimorphism, then $\rho$ is an isomorphism. 
\smallskip
In particular: If $M \in \mathcal{F}$, then also $\mathcal{A}(M) \in \mathcal{F}$.
\end{lema}

\begin{proof}  For the first implication it suffices to observe that, under the given premise, $\tau:  N \rightarrow M$ is in turn a $\pinf(\lamod)$-approximation of $M$.  Minimality of $p$ thus shows that  $\text{length}(\mathcal{A}(M)) \le \text{length}(N)$.  Therefore the epimorphism $\rho$ is an isomorphism.

To derive the final implication, assume that $M$ is torsionfree, and apply the preceding implication to the canonical map $\rho: \mathcal{A}(M) \rightarrow \mathcal{A}(M) /\Delta\bigl(\mathcal{A}(M) \bigr)$.  
\end{proof}
\smallskip

\subsection{The main theorem}

The upcoming theorem shows that, in the situation of \ref{condition4.1}, contravariant finiteness of 
$\mathcal{P}^{<\infty}(\lamod)$ in $\lamod$ is equivalent to 
contravariant finiteness of $\pinf(\elaemod)$ in $\elaemod$.   In fact, we obtain sharper information, relating minimal $\pinf$-approximations in $\lamod$ to minimal 
$\pinf$-approximations in $\elaemod$.

As before, $(\mathcal{T}, \mathcal{F}) = ({\mathcal{T}} _ {e}, {\mathcal{F}}_e)$ is the torsion pair of subsection 2.2.  Recall 
that the class $\mathcal{T}$ consists of the $\la$-modules all of whose composition factors 
belong to $\add\bigl(\Lambda (1 - e)/ J (1 - e)\bigr)$, while $\mathcal{F}$ consists of the modules $F$ 
with $\soc F \in \add(\Lambda e/ Je)$.

\begin{teor} \label{thm.reduction-to-simples-of-infteprojdim}
%\proclaim{Theorem} %
Let $\Lambda$ be a basic Artin algebra, and let $e\in\Lambda$ be an idempotent such that 
the left  $\Lambda$-module $\Lambda (1 - e)/ J(1 - e)$ has finite 
projective dimension and $e\Lambda (1 - e)$ has finite 
projective dimension as a left $e\Lambda e$-module.  

Then $\mathcal{P}^{<\infty}(\lamod)$ is contravariantly finite in $\lamod$ if and only if $\mathcal{P}^{<\infty}(\elaemod)$ is contravariantly finite in $e\Lambda e\text{-}{\mod}$. 

 In more detail:  The following two implications hold for every finitely generated left $\Lambda$-module $F \in \mathcal{F}$.

 \begin{enumerate}
\item If  $p:M\longrightarrow F$ is a minimal right  $\mathcal{P}^{<\infty}(\lamod)$-approximation of $F$, 
then $p_{| eM}:eM\longrightarrow eF$ is a minimal right $\mathcal{P}^{<\infty}(\elaemod)$-approximation of $eF$.  
\smallskip
\item Suppose that $q:X\longrightarrow eF$ is a minimal right  $\mathcal{P}^{<\infty}(e\Lambda e \text{-}\mod)$-approximation of $eF$, and let $f$ be the composition 
$$\frac{\Lambda e\otimes_{e\Lambda e}X}{\Delta (\Lambda e\otimes_{e\Lambda e}X)} \ \ \ \overset{\overline{1\otimes q}}{\longrightarrow} \ \ \ \frac{\Lambda e\otimes_{e\Lambda e}eF}{\Delta (\Lambda e\otimes_{e\Lambda e}eF)}\ \  \overset{\overline{\epsilon}_F}{\longrightarrow}\ \ F,$$ 
where $\overline{\epsilon}_F$ is induced by the natural map $\epsilon_F: \Lambda e \otimes_{e \Lambda e} eF \rightarrow F$.
Then any representative of the unique maximal element of the poset $\mathcal{E}_f$ of 
essential $\Delta$-extensions of $f$ (cf.  Proposition \ref{prop.poset of Delta-extensions} for existence) is a minimal right $\mathcal{P}^{<\infty}(\Lambda \text{-}\mod)$-approximation of $F$. 
\end {enumerate}
 
\end{teor}

\begin{rem} \label{Remark 4.7} 
 By Lemma \ref{core}, the map $\overline{1 \otimes q}$ in statement (2) of the theorem is naturally 
isomorphic to the map
 $\mathfrak{core}(1 \otimes q):  \mathfrak{core}(\Lambda e \otimes_{e \Lambda e} X) \rightarrow \mathfrak{core}(F)$ and, modulo 
the canonical isomorphism $\mathfrak{core}(F) \cong \frac{\Lambda e\otimes_{e\Lambda e}eF}{\Delta (\Lambda e\otimes_{e\Lambda e}eF)}$, the map $\overline{\epsilon}_F: \mathfrak{core}(F) = \nabla(F) \rightarrow F$ is the embedding of $\mathfrak{core}(F)$ into $F$.  
\end{rem}

\begin{proof}
We start by showing that the first assertion follows from (1) and (2).   Indeed, the correspondence $S \leftrightarrow eS$ is a bijection between the isomorphism classes of simple $\la$-modules in $\add(\Lambda e /Je)$ and those of the simple left $e \Lambda e$-modules.  In light of \cite [Proposition 3.7]{AR1} and the fact that the simples in $\add\bigl(\la (1 - e)/ J(1 - e)\bigr)$ have finite projective dimension, it thus suffices to show that any 
simple module $S \in \add(\Lambda e/ Je)$ has a $\mathcal{P}^{<\infty}(\lamod)$-approximation precisely when $eS$ has a $\pinf(\elaemod)$-approximation.  Clearly, the simple summands $S$ of $\Lambda e / Je$ satisfy $S = \Lambda e S \in \mathcal{F}$, and thus the required equivalence arises from (1) and (2) as a special case.
\smallskip

 (1) We assume that  $p:  M \rightarrow F$ is a minimal $\mathcal{P}^{<\infty}(\lamod)$-approximation of $F$, and set $p'=p_{| eM}:eM\longrightarrow eF$.  
 Moreover, we abbreviate $\Lambda e \otimes_{e \Lambda e}Y$ to $Y^\dagger
$
 for $Y \in {\elaemod}$ and $\Lambda e  \otimes_{e \Lambda e} eN$ to $N^\ddagger$ for 
$N\in \laMod$; thus $(-)^\dagger$ and $(-)^\ddagger$ are functors $\elaeMod \rightarrow 
\laMod$ and $\laMod \rightarrow \laMod$, respectively, such that $N^
\ddagger = (eN)^\dagger$.  Let $\alpha \in \Hom_{e \Lambda e}(Z, eF)$ for some $Z \in \pinf(\elaemod)$. By 
Lemma \ref{lem.preserving-reflecting finprojdim}, 
$Z^\dagger \in \mathcal{P}^{<\infty}(\lamod)$, and hence the composition $\epsilon_F \circ (1 \otimes \alpha):  \Lambda e \otimes_{e \Lambda e} Z  = Z^\dagger \rightarrow F$ belongs to $\Hom_{\Lambda}(\mathcal{P}^{<\infty}(\lamod), F)$.  By hypothesis, we thus obtain a 
map $\beta \in \Hom_{\Lambda}(Z^\dagger, M)$ such that $\epsilon_F \circ (1 \otimes \alpha) = p \circ \beta$. Since the 
functor ${\bf{e}} \circ (\Lambda e \otimes_{e \Lambda e} -) = {\bf{e}} \circ (-)^\dagger$ is naturally equivalent to $1_{{\bf{e}} \Lambda e\text{-Mod}}$, 
application of ${\bf{e}}$ to this equality shows that $\beta$ factors through $p'$.   Therefore $p'$ is a 
$\pinf(\elaemod)$-approximation of $eF$.

To prove that $p'$ is (right) minimal, it suffices to show that no nonzero direct summand of $eM$ is contained in the kernel of $p'$.  So suppose $eM=X\oplus Y$ with $p'(Y) = 0$, and let $\pi: eM \rightarrow X$ be the projection onto $X$ along $Y$.  The latter gives rise to a corresponding  projection 
${\rm {pr}}: M^\ddagger/ \Delta(M^\ddagger) \rightarrow  X^\dagger/ \Delta(X^\dagger)$. 
 Bearing 
in mind that $M \in \mathcal{F}$ by Lemma \ref{lem.P-cover of torsionfree},  we obtain the following commutative diagram with exact rows 
$$\xymatrixrowsep{2.5pc}\xymatrixcolsep{4pc}
\xymatrix{ 0 \ar[r]  &M^\ddagger/\Delta(M^\ddagger) \ar[d]_{\text{pr}} \ar[r]^-{\overline{\epsilon}_M}  &M \ar[d]^{\rho} 
\ar[r]  &\Coker(\overline{\epsilon}_M) \ar[d]^{\cong} \ar[r]  &0  \\0 \ar[r]  &X^\dagger/\Delta(X^\dagger) \ar[r]^-{u}  &N 
\ar[r]  &\Coker(u) \ar[r]  &0}$$
\noindent where the left-hand square is the pushout of $\overline{\epsilon}_M$ and $\text{pr}$.  In particular, $\rho$  is an epimorphism and $u$ is a monomorphism. Both of the flanking terms of the bottom exact sequence have 
finite projective dimension:  For $X^\dagger/ \Delta(X^\dagger)$ this follows from Lemma \ref{lem.preserving-reflecting finprojdim}, for the righthand term it follows from  $\Coker(u) \cong \Coker(\overline{\epsilon}_M) \cong M/ \nabla(M)$. Consequently, also $N$ has finite projective dimension. 

On defining $q$ to be the restriction of $\overline{\epsilon}_F \circ (\overline{1 \otimes p'})$ to $X^\dagger/ \Delta(X^\dagger)$, one readily checks that the following diagram commutes: 
$$\xymatrixrowsep{2.5pc}\xymatrixcolsep{4pc}
\xymatrix{M^\ddagger/\Delta(M^\ddagger) \ar[d]_{\text{pr}} \ar[r]^-{\overline{\epsilon}_M}  &M \ar[d]^{p}  \\X^\dagger/\Delta(X^\dagger) \ar[r]^-{q}  &F}$$
\noindent The universal property of the pushout therefore yields a map $v \in \Hom_{\la}(N, F)$ such that $v \circ \rho = p$. By Lemma \ref{lem.P-cover of torsionfree}, $\rho$ is an isomorphism, whence so is $\text{pr}$, meaning that $Y^\dagger/ \Delta(Y^\dagger) = 0$.  Thus $Y^\dagger \in  \mathcal{T}$, and we conclude that $Y \cong {\bf{e}}(Y^\dagger) = 0$.  This confirms minimality of $p'$. 

\smallskip 

(2) Let $q:X\longrightarrow eF$ and $f$ be as in the assertion, and suppose that the pair $(M, p)$ represents the maximum of the set $\mathcal{E}_f$.  To check that $p: M \rightarrow F$ is a $\mathcal{P}^{<\infty}(\lamod)$-approximation of $F$, let 
$h \in \Hom_{\Lambda}(N,F)$ with
 $N \in \mathcal{P}^{<\infty}(\lamod)$.  Since $\Delta(F) = 0$ 
 and $N/ \Delta(N)$ again belongs to $\mathcal{P}^{<\infty}(\lamod)$, we may assume that $\Delta(N) = 0$.  In light of Lemma \ref{lem.preserving-reflecting finprojdim}, the restriction $h' = h|_{eN}: eN \rightarrow eF$ is a morphism in $\Hom_{e \Lambda e}\bigl(\pinf(\elaemod), eF\bigr)$, whence it factors through $q$, say $h' = q \circ \eta$ with $\eta \in \Hom_{e\Lambda e}(eN, X)$.  Again referring to Lemma \ref{lem.preserving-reflecting finprojdim}, we moreover see that the map $\overline{1 \otimes \eta}: N^{\ddagger}/ \Delta(N^{\ddagger}) \rightarrow X^{\dagger}/ \Delta(X^{\dagger})$ induced by $\eta$ is a morphism in 
$\mathcal{P}^{<\infty}(\lamod)$.  In order to suitably extend $\overline{1 \otimes \eta}$ to a homomorphism with domain $N$ via the monomorphism $\overline{\epsilon}_N: N^{\ddagger}/ \Delta(N^{\ddagger}) \rightarrow N$, we consider the pushout diagram of the pair $(\overline{\epsilon}_N, \overline{1 \otimes \eta})$, as shown in the diagram below:  
$$\xymatrixrowsep{2.5pc}\xymatrixcolsep{4pc}\xymatrix{0 \ar[r]  &N^\ddagger/\Delta(N^\ddagger) \ar[d]_{\overline{1\otimes\eta}} \ar[r]^-{\overline{\epsilon}_N}  &N \ar[d]_{g} \ar[r] 
\ar@'{@+{[0,0]+(8,-5)} @+{[1,0]+(8,0)} @+{[2,0]+(8,0)}}[dddl]^(0.75){h}  &\Coker(\overline{\epsilon}_N) 
\ar[d]^{\cong} \ar[r]  &0  \\
0 \ar[r]  &X^\dagger/\Delta(X^\dagger) \ar[d]^{\overline{1\otimes q}} \ar[r]^-{\lambda} \ar@/_3pc/[dd]_{f} &N' \ar[r]  &
\Coker(\lambda) \ar[r]  &0 \\ &F^\ddagger/\Delta(F^\ddagger) \ar[d]^{\overline{\epsilon}_F}    &  \\ &F }$$
\noindent Since the image of $\overline{\epsilon}_N$ is $\nabla (N)$, we may use the argumentation in the proof of (1) to infer 
that $N' \in \mathcal{P}^{<\infty}(\lamod)$.  
By the naturality of $\overline{\epsilon}$ 
(see Lemma \ref{core}),  $\overline{\epsilon}_F \circ \overline{1 \otimes h'}$ 
coincides with $h \circ \overline{\epsilon}_N$, which yields $f \circ (\overline{1 \otimes \eta}) = h \circ 
\overline{\epsilon}_N$; in other words, the above diagram fully commutes.  Thus the universal property of the pushout 
provides us with a map $\phi \in \Hom_{\Lambda}(N', F)$ such that $\phi  \circ \lambda  =f$ 
and $\phi \circ g = h$.  Given that 
$F$ is torsionfree, $\phi$ factors through the canonical map $\pi: N' \rightarrow N'/ \Delta(N')$; denote the induced 
map $N'/ \Delta(N') \rightarrow F$ by $\overline{\phi}$, and note that the composition $\overline{\lambda}: = \pi \circ 
\lambda$ is still an embedding, since Im$(\lambda) \cap  \Delta(N')  = 0$; it is thus harmless to view $\overline{\lambda}$ as a 
set inclusion. We will ascertain that  the pair $\bigl(N'/\Delta(N'),\, \overline{\phi}\, \bigr)$ gives rise to a class in
 $\mathcal{E}_f
$:  Indeed, the cokernel of $\overline{\lambda}$ belongs to $\mathcal{T}$, since it is an 
epimorphic image of $N / \nabla(N)
$.  Torsionfreeness of $N'/ \Delta(N')$ moreover guarantees that $\overline{\lambda}$ is an essential extension.   
That $\overline{\phi}$ extends $f$, is immediate from our construction.  
The inequality $[\bigl(N'/ \Delta(N'),\, \overline{\phi}\, \bigr )] \preceq  [(M, p)]$ now yields a monomorphism $v:  N'/
\Delta(N') \rightarrow M$ such that $p \circ v = \overline{\phi}$.  It is straightforward to deduce that  $h$ factors 
through $p$ as required.   
 
To see that $p$ is a right minimal morphism, consider a decomposition
 $p = \begin{pmatrix}  p_1& 0 \end{pmatrix}   :M=U \oplus V\longrightarrow F$, where $p_1$ is right minimal. Due to the fact that $[(M,p )]\in\mathcal{E}_f$, the domain $M$ of $p$ is an essential extension of $X^{\dagger}/ \Delta(X^{\dagger})$ such that the quotient of $M$ modulo $X^{\dagger}/ \Delta(X^{\dagger})$ belongs to $\mathcal{T}$.  This means that $eM = e \bigl(X^{\dagger}/ \Delta(X^{\dagger})\bigr)$, whence the induced 
{map $p'=p_{| eM}:eM\rightarrow eF$ } coincides with the restriction $f'$ of $f$ to $e \bigl(X^{\dagger}/ \Delta(X^{\dagger})\bigr)$. In light of the natural equivalence 
${\bf{e}} \circ (\Lambda e \otimes_{e \Lambda e} -)  \cong 1_{\elaemod}$, this identifies $f'$ with the minimal 
approximation $q:X\rightarrow eF$.  Due to the matrix decomposition $p'=\begin{pmatrix} p_1' & 0 \end{pmatrix} :X\cong eU\oplus eV\rightarrow eF$, where $p_1' = p_1 {_{| eU}}$,  right minimality of $q$ thus yields $eV=0$, so that $V\in\mathcal{T}$.  On the other hand, $V \in \mathcal{F}$; this follows from the fact that $M \in \mathcal{F}$, because the torsionfree module $X^{\dagger} / 
\Delta(X^{\dagger})$ is an essential submodule of $M$.  Consequently, $V=0$, which proves minimality of $p$ as claimed. 
\end{proof}

We apply Theorem \ref{thm.reduction-to-simples-of-infteprojdim} 
to the situation where the $\pinf$-categories of $\elaemod$ and (equivalently) $\lamod$
 are contravariantly finite in the corresponding ambient module categories.  The proposition below picks up the theme of Lemma \ref{Lemma 4.4} and Remark \ref{Remark 4.7}. 
 It reinforces our understanding of the links among the minimal approximations of objects 
$M \in \lamod$ 
 and those of the corresponding $\la$-modules $M/\Delta(M)$, $M_\sigma$, and $\mathfrak{core}(M)$. 
  These connections underlie the upcoming exploration of the basic algebra $\latilde$ that results from strongly tilting $\lamod$.

\begin{prop} \label{prop.explicit calculation of approximation}
Continue to adopt conditions {\rm 4.1(i)} and {\rm (ii)}, and suppose that $\pinf(\elaemod)$ is contravariantly finite in $\elaemod$. {\rm $\bigl($By Theorem \ref{thm.reduction-to-simples-of-infteprojdim}, this implies contravariant finiteness of $\pinf (\lamod) $ in $\lamod$.$\bigr)$} For $M \in \lamod$, denote by $\mathcal{A} (M)$ the minimal right $\pinf(\lamod)$-approximation of $M$.
\smallskip

\noindent {\bf (a)} If $M \in \mathcal {G}$, then $\mathcal{A}(M) \in \mathcal{G}$.
\smallskip

\noindent {\bf (b)}  If $M$ is torsionfree, then $\mathcal {A} (M)$ naturally 
embeds into $\mathcal {A} (M_\sigma)$, and the core of $\mathcal{A} (M)$ coincides with that of
 $\mathcal{A}(M_\sigma)$.
\smallskip

\noindent {\bf (c)}  For arbitrary 
$M \in \lamod$, $\mathcal{A}  \bigl(M/\Delta(M)\bigr) \cong \mathcal{A} (M)/\Delta(\mathcal{A}(M))$, and the torsion submodule $\Delta(\mathcal{A} (M))$ is isomorphic to $\Delta(M)$. 
In particular, there is a natural homomorphism 
$\rho: \mathcal{A} (M) \rightarrow \mathcal{A} (M_\sigma)$ such that 
 $\Ker(\rho)$ and   $\Coker (\rho)$ belong to $\mathcal{T}$.  
\end{prop}
\begin{proof} Part (a) is a consequence of Lemma 4.4.  We re-encounter this fact in a notational setup that is convenient towards parts (b) and (c).

Suppose that $M \in \lamod$ is torsionfree, and fix an injective envelope $E$ of $\mathcal{A}(M)$.  Since $\mathcal{A}(M)$ is again torsionfree 
(Lemma 4.5), so is $\soc\bigl(\mathcal{A}(M)\bigr)$, which implies that $\soc\bigl(\mathcal{A}(M) \bigr)$ is contained in the submodule $\mathfrak{core}\bigl(\mathcal{A}(M)\bigr)$ of $\mathcal{A}(M)$.  
Therefore $E$ is also an injective envelope of $\mathfrak{core}\bigl(\mathcal{A}(M)\bigr)$.
  Moreover, let $p: \mathcal{A} (M) \rightarrow M$ be a minimal $\pinf(\lamod)$-approximation, and set $X = e \mathcal{A}(M)$. 
By Theorem 4.6 and Remark 4.7, the restriction ${\bf e}(p): e \mathcal{A}(M) \rightarrow eM$ is then a minimal 
$\pinf(\elaemod)$-approximation of $eM$, and  $p$ is maximal in the set $\mathcal{E}_f$ of essential $\Delta$-extensions of the composition 
\begin{equation*}
f: \  \frac{\Lambda e\otimes_{e\Lambda e}X}{\Delta (\Lambda e\otimes_{e\Lambda e}X)}\, \cong\, \, \mathfrak{core}(\mathcal{A} (M)) \  \longrightarrow\,  \mathfrak{core}(M) \hookrightarrow \ M.   \tag{$\bullet$}
\end{equation*}

{ (a)} Suppose $M = M_\sigma$.  In view of the fact that
 ${\bf e}(p): X = e  \mathcal{A} (M) \rightarrow eM$ is a minimal $\pinf(\elaemod)$-approximation of $eM$, Lemma 4.4 guarantees that $(\Lambda e \otimes_{e \Lambda e} X)_\sigma$ is a minimal $\pinf(\lamod ) $-approximation of 
$M_\sigma = M$, i.e., $ \mathcal{A} (M)  \cong \bigl( \mathcal{A} (M))_\sigma$ as claimed.

{ (b)}  In light of (a), $(\Lambda e \otimes_{e \Lambda e} X)_\sigma \cong \bigl(\frac{\Lambda e\otimes_{e\Lambda e}X}{\Delta (\Lambda e\otimes_{e\Lambda e}X)}\bigr)_\sigma$ is a minimal 
$\pinf(\lamod)$-approximation of $M_\sigma$. Hence we infer from $(\bullet)$ that 
$$ \mathcal{A}  (M_\sigma) \cong \bigl(\mathfrak{core} \mathcal{A} (M)\bigr)_\sigma \cong   \mathcal{A} (M)_\sigma.$$We may thus identify $ \mathcal{A} (M_\sigma)$ with a submodule of $E$ containing $ \mathcal{A} (M)$ such that $ \mathcal{A} (M_\sigma) /  \mathcal{A} (M) \in \mathcal{T}$.  The latter condition implies that 
$\mathfrak{core} \bigl( \mathcal{A} (M_\sigma)\bigr) = \mathfrak{core} \bigl( \mathcal{A} (M)\bigr)$.  

{(c)}  Now let $M$ be arbitrary, and let $\pi:  M \twoheadrightarrow \overline{M} = M / \Delta(M)$ be the canonical map.  Moreover, suppose that $q:  \mathcal{A} (\overline{M}) \rightarrow \overline{M}$ is a minimal $\pinf(\lamod)$-approximation of $\overline{M}$.  On the model of previous arguments, it is straightforward to check that the map $q': N \rightarrow M$ in the following pullback is a
 $\pinf(\lamod)$-approximation of $M$.
$$\xymatrixrowsep{2.5pc}\xymatrixcolsep{4pc}
\xymatrix{N \ar[d]_{\pi'} \ar[r]^-{q'}  &M \ar[d]^{\pi}  \\  \mathcal{A}(\overline{M}) \ar[r]^-{q}  &\overline{M}}$$

The fact that $\Ker(\pi') \cong \Ker (\pi)$ shows that $\Ker(\pi')$ is a torsion module, necessarily equal to $\Delta(N)$ because $ \mathcal{A} (\overline{M})$ is torsionfree; hence $q'$ induces an isomorphism 
$\Delta(q'): \Delta(N) \cong \Delta(M)$.  To check right minimality of $q'$, suppose
 $u \in \End_{\Lambda}(N)$ satisfies $q' \circ u = q'$.  On applying the functors $\Delta$ and $(1:\Delta)$, the latter acting as $M\rightarrow (1:\Delta )(M):=\overline{M}$, we find that $\Delta(u) \in \End_{\Lambda}\bigl(\Delta(N)\bigr)$ is an isomorphism, since $\Delta(q')$ is an isomorphism; further $(1: \Delta)(u) \in \End_{\la}\bigl(N / \Delta(N)\bigr) \cong \End_{\la}\bigl( \mathcal{A} (\overline{M})\bigr)$ is an isomorphism, since $(1: \Delta )(q') \cong q:  \mathcal{A}  (\overline{M}) \rightarrow \overline{M}$ was chosen right minimal.  Consequently, also $u$ is an isomorphism, which proves $q$ to be right minimal as claimed.

We conclude that $q'$ coincides with $p$ up to isomorphism, whence 
$  \mathcal{A} (M) / \Delta( \mathcal{A}  (M))\linebreak {\cong  \mathcal{A} (\overline{M})}$ and $\Delta\bigl( \mathcal{A} (M)\bigr) \cong \Delta(M)$.  In particular, the composition $$\rho:  \mathcal{A}  (M)\  \twoheadrightarrow\   \mathcal{A} (M)/ \Delta(M)\   \cong\   \mathcal{A} (\overline{M}) \ \hookrightarrow \  \mathcal{A}  (\overline{M}_\sigma)  =  \mathcal{A} (M_\sigma)$$ 
has the postulated properties. 
\end{proof}

\begin{rem} \label{rem.conflicting one}
Analyzing the proof of the preceding proposition, one observes the following: If $\pi_X:X\longrightarrow\overline{X}=X/\Delta (X)$ and $p_X:\mathcal{A}(X)\longrightarrow X$ denote the canonical projection and the minimal $\pinf( \lamod)$-approximation of a module $X$, then one has a composition of pullbacks 
$$\xymatrixrowsep{2.5pc}\xymatrixcolsep{4pc}
\xymatrix{ \mathcal{A} (M)  \ar[d]_{\pi _{\mathcal{A} (M)}} \ar[r]_{p_M}   &M \ar[d]^{\pi _M}  \\  \mathcal{A}(\overline{M}) \ar[d]
 \ar[r]_{p_{\overline{M}}}  & \overline{M} \ar[d]   \\
\mathcal{A} (M_{\sigma} ) \ar[r]_{p_{M_{\sigma}} }   & M_{\sigma}   }
$$
such that the compositions of the vertical arrows are  isomorphic to the canonical maps $\mu_{\mathcal{A}(M)}:\mathcal{A}(M)\longrightarrow\mathcal{A}(M)_\sigma$ and $\mu_M:M\longrightarrow M_\sigma$, respectively.
\end{rem}

In applying Theorem \ref{thm.reduction-to-simples-of-infteprojdim}, we typically decompose $e$ and $1 - e$ into primitive
 idempotents: $e = e_1 + \cdots + e_m$ and $1-e = e_{m+1} + \cdots + e_n$.  
Since part (2) of Theorem \ref{thm.reduction-to-simples-of-infteprojdim} aims at reducing the contravariant finiteness test 
for $\mathcal{P}^{<\infty}(\lamod)$ to that for $\pinf(\elaemod)$, we are interested in making  $e \Lambda e$ as ``small" as possible.  Hence the situation where {\it all\/}  simple left modules $S_i$ of finite projective dimension correspond to idempotents $e_i$ for $i \geq m +1$ is of particular interest.  A modification of an example by Fuller-Saor\'\i n (see \cite [Example 4.2]{FS}) shows that, even for monomial algebras $\la$, neither of the conditions \ref{condition4.1}(i) nor \ref{condition4.1}(ii) in the hypothesis of Theorem \ref{thm.reduction-to-simples-of-infteprojdim} is dispensable. 

\begin{ejem}
 Let $\la$ be the monomial algebra based on the quiver 
$$\xymatrixcolsep{4pc}
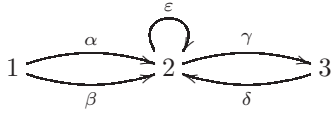
\xymatrix{
1 \ar@/^/[r]^{\alpha} \ar@/_/[r]_{\beta} &2 \ar@(ul,ur)^{\varepsilon} \ar@/^/[r]^{\gamma} &3 \ar@/^/[l]^{\delta}
}$$
which is defined by the graphs of its indecomposable projective left $\la$-modules, namely
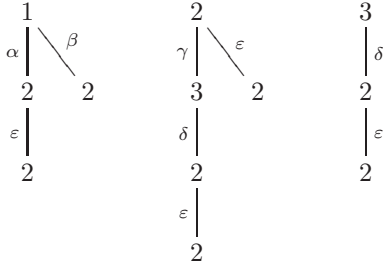
\[
\xymatrix@R=1.5pc@C=1pc{%
	1 \ar@{-}[d]_{\alpha} \ar@{-}[dr]^{\beta} &&&
	2 \ar@{-}[d]_{\gamma} \ar@{-}[dr]^{\varepsilon} &&&
		3 \ar@{-}[d]^{\delta} \cr
	2 \ar@{-}[d]_{\varepsilon} & 2 &&
		3 \ar@{-}[d]_{\delta} &2 &&2 \ar@{-}[d]^{\varepsilon} \cr
	2 &&&2 \ar@{-}[d]_{\varepsilon} &&&2  \cr
&&&2
}
\]

Clearly, the simple left $\la$-modules $S_1$, $S_2$ have infinite projective dimension, whereas
 $\pdim_{\la} S_3 =2$.  
Let $e = e_1 + e_2$.  Then condition \ref{condition4.1}(i) is satisfied.  On the other hand, the indecomposable projective left $e \Lambda e$-modules $e \Lambda e_i$ for $i = 1,2$ and the $e \Lambda e$-module $e \Lambda (1 - e) = \Lambda e_3$ have the following graphs:
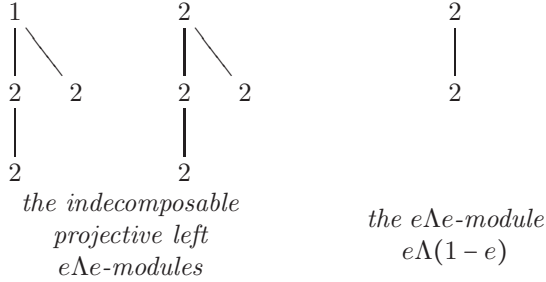
\[
\begin{matrix}
	\xymatrix@R1.5pc@C1pc{%
		1 \ar@{-}[d] \ar@{-}[dr] &&&2 \ar@{-}[d] \ar@{-}[dr]  \cr
		2 \ar@{-}[d] &2 &&2 \ar@{-}[d] &2  \cr
		2 &&&2
	}
	& \qquad &
	\xymatrix@R1.5pc@C1pc{%
		2 \ar@{-}[d]  \cr
		2
		}  \cr
	\begin{matrix}
		\hbox{the indecomposable}\cr
		\hbox{projective left}\cr
		e\Lambda e\hbox{-modules}
	\end{matrix}
	&&
	\begin{matrix}
		\hbox{the}\ e\Lambda e\hbox{-module} \cr
		e\Lambda(1-e)
	\end{matrix}
\end{matrix}
\]

This shows that the left $e \Lambda e$-module $e \Lambda (1- e)$ has infinite projective dimension.  To see that the equivalence of Theorem \ref{thm.reduction-to-simples-of-infteprojdim} fails, observe that  $\lfindim e\Lambda e = 0$, whence $\pinf(\elaemod)$ is contravariantly finite in $\elaemod$.
  Yet $\mathcal{P}^{<\infty}(\lamod)$ is not contravariantly finite in 
$\lamod$, since the simple module $S_1$ does not have a $\mathcal{P}^{<\infty}(\lamod)
$-approximation.  Indeed, consider the family 
$(M_n)_{n \in \Bbb{N}}$
 of objects in $\mathcal{P}^{<\infty}(\lamod)$ shown below:

$$
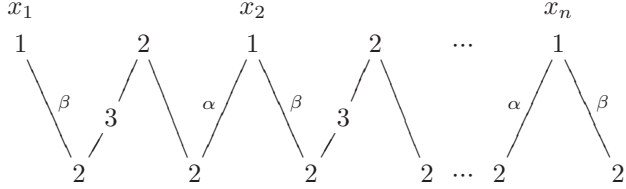
\xymatrix@R.75pc@C.1pc{%
	\overset{\displaystyle x_1\smash[t]{\vphantom\int}}{1} \ar@{-}[ddrr]^{\beta} &&&&
		2 \ar@{-}[dl] \ar@{-}[ddrr] &&&&
		\overset{\displaystyle x_2\smash[t]{\vphantom\int}}{1} \ar@{-}[ddll]_{\alpha} \ar@{-}[ddrr]^{\beta} &&&&
		2 \ar@{-}[dl]  \ar@{-}[ddrr] &&& \cdots &&&
		\overset{\displaystyle x_n\smash[t]{\vphantom\int}}{1} \ar@{-}[ddll]_{\alpha} \ar@{-}[ddrr]^{\beta} \cr
	&&&
		3 \ar@{-}[dl] &&&&&&&&
		3 \ar@{-}[dl] &&&&&&  \cr
	&& 2 &&&&
		2 &&&& 2 &&&& 2 & \cdots & 2 &&&& 2
}
$$
\smallskip

\noindent With the aid of Criterion 10 of \cite{Happel-HZ},  one readily checks that no
 homomorphism $\phi \in\Hom_{\la}\bigl(\mathcal{P}^{<\infty}(\lamod), S_1\bigr)$  allows for 
factorization of {\it all\/} of the following maps $f_n \in \Hom_{\la}(M_n, S_1)$ in the form $f_n = \phi \circ g_n$; here 
$f_n(x_1) = \overline{e_1} \in \Lambda e_1/ J e_1$ and $f_n(x_i) = 0$ for $i > 1$.  
\end{ejem}

We mention that, in the presence of condition \ref{condition4.1}(ii), condition \ref{condition4.1}(i) in the hypothesis of Theorem \ref{thm.reduction-to-simples-of-infteprojdim} is not superfluous either.  Instances attesting to this are ubiquitous.  In the above example, take $e = e_1$, and note that this choice makes $e \Lambda  e$ semisimple.  

\section{  The basic strong tilting object in $\mathcal{P}^{<\infty}(\lamod)$ and its endomorphism algebra }

In this section, we focus on the situation where $\mathcal{P}^{<\infty}(\lamod)$ 
is contravariantly finite in $\lamod$. Let $\latilde = \End_{\la}(T)^{\op}$, where $_{\la} T$ is the basic 
strong tilting object in $\lamod$.  The guiding question is this:  When is $\pinf(\modlatilde)$ in 
turn contravariantly finite in $\modlatilde$?  In light of Theorem \ref{teor.strong-tilting-iteration},
 this amounts to the question of when the 
process of strongly tilting $\lamod$ can be repeated arbitrarily.  As witnessed by Example \ref{ejem.Igusa-Smalo-Todorov}, the 
possibility of iteration is not automatic in case $\lamod$ can be strongly tilted once.

\begin{setting} [Upgraded blanket hypothesis] \label{5.1}
Throughout this section, let $\la$ be a basic Artin algebra and $e \in \la$ an idempotent satisfying 
conditions (i)  and (ii) of Setting  \ref{condition4.1}, namely $\pdim_{\la} \Lambda (1 - e)/ J(1 - e) < \infty$ and $\pdim_{e \Lambda e} e \Lambda (1- e) < \infty$, next to the additional condition that 
\smallskip

\noindent  (iii) $\pinf(\elaemod)$ is contravariantly finite in $\elaemod$.
\end{setting}

By Theorem \ref{thm.reduction-to-simples-of-infteprojdim}, $\mathcal{P}^{<\infty}(\lamod)$ is contravariantly 
finite in $\lamod$ in this setting.  We will introduce an idempotent $\etilde$ in $\latilde$ which naturally corresponds to $e$. Our objective is to show that the blanket hypotheses \ref{condition4.1} of the previous section carry over to $\modlatilde$, meaning that the right $\latilde$-module $(1 - \etilde) \latilde/ (1 - \etilde) \Jtilde$ in turn has finite projective dimension, and the corner $(1 - \etilde) \latilde \etilde$ has finite projective dimension 
as a right $\etilde \latilde \etilde$-module.  In light of Theorem \ref{thm.reduction-to-simples-of-infteprojdim}, this will then allow us to deduce contravariant finiteness of $\pinf(\modlatilde)$ from that of $\pinf(\modetillatiletil)$.  

To this end, we will first assemble some information about the $\la$-module structure of the basic strong tilting module $_{\la} T$.

  From Section 2.1 we know that $\add T = \add \mathcal{A}$, where $\mathcal{A}$ is 
a minimal $\mathcal{P}^{<\infty}(\lamod)$-approximation of an injective
 cogenerator of $\lamod$.  In particular, any such minimal approximation $\mathcal{A}$ 
is a strong tilting module.  To pin down a candidate for $\mathcal{A}$, 
we decompose both $e$ and $1 - e$ into sums of primitive idempotents of $\la$, say
 $e = e_1 + \cdots + e_m$ and $1- e = e_{m+1} + \cdots + e_n$.  Moreover, we let $S_i = \Lambda e_i/ J e_i$ be the corresponding simple modules, and choose $\mathcal{A}_i$ to be the minimal 
$\mathcal{P}^{<\infty}(\lamod)$-approximations of their injective envelopes $E(S_i)$.  Clearly, 
$\mathcal{A} := \bigoplus_{1 \le i \leq n} \mathcal{A}_i$ is then as required. 
 Since the injective objects of the Giraud subcategory $\mathcal{G}$ of 
$\lamod$ relative to the torsion pair $(\mathcal{T}, \mathcal{F})$ 
(see 2.4) are precisely the torsionfree injective $\la$-modules, the subsum $\bigoplus_{1 \leq i \leq m} E(S_i)$ is an injective cogenerator in $\mathcal{G}$.  
Moreover, according to Lemma 2.2, left multiplication by $e$ induces an equivalence of
 categories $\mathcal{G} \cap \lamod \cong (e \Lambda e)\text{-mod}$, and $\bigoplus_{1 \le i \le m} e E(S_i)$ is an injective cogenerator in $\elaemod$. 
  We separately explore the direct sums $\bigoplus_{1 \le i \le m} \mathcal{A}_i$ and
 $\bigoplus_{m+1 \le i \le n} \mathcal{A}_i$.
\smallskip

\begin{prop} \label{prop.decomposition-strongtilting}
 Assume the hypotheses \ref{5.1}, and let $T \in \lamod$ be the basic strong tilting module. 
  Then $\add(\bigoplus_{1 \le i \le m} \mathcal{A}_i)$
 contains precisely $m$ isomorphism classes of indecomposable modules,
 represented by $T_1, \dots, T_m$ say. If $T' =  \bigoplus_{1 \le i \le m} T_i$,  then the direct summand
 $T'$ of $T$ is an object of \, $\mathcal{G}$, and $e T'$ is the basic 
strong tilting object in $\elaemod$, up to isomorphism.  

Supplementary direct summands $T_{m+1}, \dots, T_n$ of $T$ such that $T = T'\, \oplus\, \bigoplus_{m+1 \le i \le n}T_i$ are as follows:  For $m+1 \le i \le n$,  the
 minimal $\mathcal{P}^{<\infty}(\lamod)$-approximation $\mathcal{A}_i$ 
of $E(S_i)$ decomposes in the form $\mathcal{A}_i = T_i \oplus U_i$ with $\Delta(\soc T_i) = S_i$ 
and $U_i \in \add(T')$.
\smallskip

\noindent In particular, $\Delta(\soc T) = S_{m+1} \oplus \cdots \oplus S_n$.  
\end{prop}

\begin{proof}  By Theorem \ref{thm.reduction-to-simples-of-infteprojdim}(1), 
$\bigoplus_{1 \le i \le m} e \mathcal{A}_i$ is a 
$\pinf(e\Lambda e\text{-mod})$-approximation of the injective cogenerator 
$\bigoplus_{1 \le i \le m} e E(S_i)$ in $\elaemod$, which shows 
that $\bigoplus_{1  \le i \le m} e \mathcal{A}_i$ is a strong tilting 
object in $\elaemod$.  
Since the rank of $K_0(e \Lambda e)$ equals $m$, this implies that $\add (\bigoplus_{1  \le i \le m} e \mathcal{A}_i)$
 contains precisely $m$ isomorphism classes of indecomposable $e \Lambda e$-modules (see Section 2.1). 
 According to Proposition  \ref{prop.explicit calculation of approximation}, the sum $\bigoplus_{1 \le i \le m} \mathcal{A}_i$ in turn belongs to $\mathcal{G}$,
 and therefore the category equivalence
 $\mathcal{G} \cap \lamod \cong e \Lambda  e\text{-mod}$
 guarantees that $\add(\bigoplus_{1 \le i \le m} \mathcal{A}_i)$
 contains the same number of isomorphism classes of indecomposable objects.  
If these are represented by  $T_1, \dots, T_m$, then $e (\bigoplus_{1 \le i \le m} T_i)$ is the basic strong tilting $e \Lambda e$-module by construction.  This proves the first claim.

For the final claims, let $i \ge m+1$. We observe that $S_i = \soc E(S_i)$ is the only simple torsion module which embeds into $\mathcal{A}_i$; indeed, 
this is immediate from the fact that $\Delta(\mathcal{A}_i) \cong \Delta\bigl(E(S_i)\bigr)$ by 
Proposition \ref{prop.explicit calculation of approximation}(c).  Hence precisely one of the indecomposable direct summands of $\mathcal{A}_i$
 contains a copy of $S_i$ in its socle; say $\mathcal{A}_i = T_i \oplus U_i$ with
 $T_i$ indecomposable and $S_i = \Delta(\soc T_i)$, while $\Delta(U_i) = 0$.  Then $\bigoplus_{m+1 \le i \le n} T_i$ consists of $n - m$ pairwise nonisomorphic direct summands in $\add(\mathcal{A})$,
 none of which occurs among $T_1, \dots, T_m$, and we conclude that $T = \bigoplus_{1 \le i \le n} T_i$ up to isomorphism.  Since the $U_i$ are torsionfree and hence do not belong to $\add(\bigoplus_{m+1 \le i \le n} T_i)$, they belong to $\add (\bigoplus_{1 \le i \le m} T_i)$. 

In view of the fact that $T'$ is torsionfree, the final claim follows.
 \end{proof} 
 %%%%%%%%%%%%%%%%%%

\begin{defi*}
{\rm
  We refer to Proposition \ref{prop.decomposition-strongtilting}. 
 Let $\etilde_i \in \latilde$ be the projection $T \rightarrow T_i \subseteq T$ with respect to the decomposition $T = \bigoplus_{1 \le i \le n} T_i$.  Viewing $\etilde_i$ as an endomorphism of $T$, 
we thus obtain a primitive idempotent in $\latilde$.  We  define $\etilde: = \sum_{1 \le i \le m} \etilde_i \in \latilde$.  Thus $\etilde$ is the projection $T \rightarrow T' \subseteq T$ along $T'' : = \bigoplus_{m+1 \le i \le n} T_i$.  
}
\end{defi*}

\begin{ejem} \label{ex.second return to Ex 2.2}
{\rm(Return to Example \ref{ex.delta.nabla}.)}  Let $\Lambda$ be the algebra of Example \ref{ex.delta.nabla}.  It is readily seen that  $e = e_1 + e_2$ satisfies the conditions of Setting \ref{5.1},  that $e\la e$ is a self-injective algebra and that the indecomposable injective modules $I_1$ and $I_2$ have finite projective dimension.  Thus $T' = T_1 \oplus T_2 = I_1 \oplus I_2$.  Recall  from Proposition \ref{prop.decomposition-strongtilting}  that $T_3$ and $T_4$ are the indecomposable summands of $\mathcal{A}_3\oplus\mathcal{A}_4$ which are not in $\mathcal{F}$. Let us fix  $j\in\{3,4\}$ in the sequel. By Proposition \ref{prop.explicit calculation of approximation} and Remark \ref{rem.conflicting one}, in order to calculate  $\mathcal{A}_j$,  we need to identify the minimal $\mathcal{P}^{<\infty}(\la\mod)$-approximation of $(I_j)_\sigma=S_2\oplus S_2$ (see Example \ref{ex.first return to Ex 2.2}).  Then, by Lemma \ref{Lemma 4.4}, such an approximation is of the form $(1\otimes q)_\sigma\oplus (1\otimes q)_\sigma$, where $q:X\longrightarrow eS_2$ is the minimal   $\mathcal{P}^{<\infty}(e\la e\mod)$-approximation, whence the projective cover $e\la e_2\longrightarrow eS_2$ since $e\la e$ is self-injective. It easily follows that $(1\otimes q)_\sigma $ gets identified with the canonical projection $\rho:(\Lambda e_2)_\sigma =I_1\longrightarrow S_2$. By Proposition \ref{prop.explicit calculation of approximation} and Remark \ref{rem.conflicting one} again, the $\la$-module $\mathcal{A}_j$ is the upper left corner of the pullback of $\rho\oplus\rho: I_1\oplus I_1\longrightarrow S_2\oplus S_2$ and the projection $\mu_{I_j}:I_j\longrightarrow (I_j)_\sigma =S_2\oplus S_2$. 
It then follows that $T_3=\mathcal{A}_3$ is given by the following diagram, and $T_4 = \mathcal{A}_4$ is obtained from $T_3$ by factoring out the copy of $S_3$ in the socle of $T_3$.
$$
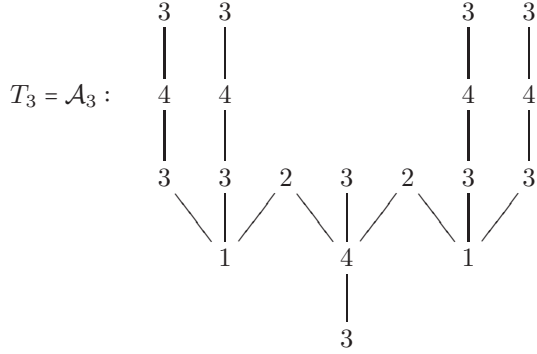
\xymatrixrowsep{1.5pc}\xymatrixcolsep{1pc}
\xymatrix{ 
&3\edge[d] &3\edge[d]  & & & &3\edge[d]  &3\edge[d]  \\
T_3 = \mathcal{A}_3:  &4\edge[d] &4\edge[d]  & & & &4\edge[d]  &4\edge[d]  \\
&3\edge[dr] &3\edge[d]  &2\edge[dr]\edge[dl]&3\edge[d]&2\edge[dr]\edge[dl]&3\edge[d]  &3\edge[dl]  \\
&&1  &&4 \edge[d]  &&1  &  \\
& &  & & 3 & &  &  
}$$
\end{ejem}

%$$
%\xymatrixrowsep{1.75pc}\xymatrixcolsep{1pc}
%\xymatrix{ 
%&3\edge[d] &3\edge[d]  & & & &3\edge[d]  &3\edge[d]  \\
%\mathcal{A}_4=T_4 = T_3/ S_2:  &4\edge[d] &4\edge[d]  & & & &4\edge[d]  &4\edge[d]  \\
%&3\edge[dr] &3\edge[d]  &2\edge[dr]\edge[dl]&3\edge[d]&2\edge[dr]\edge[dl]&3\edge[d]  &3\edge[dl]  \\
% &&1  &&4  &&1  &  
%} $$

%\noindent  The idempotent $\etilde \in \latilde$ introduced above equals the canonical projection of $T= \bigoplus_{1 \le i \le 4} T_i$ onto $T_1 \oplus T_2$.

\begin{teor}\label{Birge5.3}
  Let $\la$ be a basic Artin algebra.  Assume the hypotheses \ref{5.1},  
and let $T$ be the basic strong tilting module in $\lamod$.  If $\latilde = \End_{\la}(T)^{\op}$, 
the corner algebra $\etilde \latilde \etilde$ is isomorphic 
to $\bigl(\End_{e \Lambda e}(e T')\bigr)^{\op}$, where $eT'$ is the basic strong tilting module in $\elaemod$ specified in {\rm Proposition \ref{prop.decomposition-strongtilting}}.

The pair $(\latilde, \etilde)$ satisfies the  right-hand versions of conditions \ref{condition4.1}:  Namely, $(i)$ the right $\latilde$-module $ (1 - \etilde) \latilde/ (1 - \etilde)\Jtilde$ has finite projective dimension, and $(ii)$ the right $\etilde \latilde \etilde$-module $(1 - \etilde) \latilde \etilde$ has finite projective dimension.
\end{teor}

\begin{proof}
The isomorphism $\etilde \latilde \etilde \cong \End_{\la}(T')^{\op} \cong \bigl(\End_{e \Lambda e} (e T')\bigr)^{\op}$ is 
immediate from Proposition \ref{prop.decomposition-strongtilting} and the comments that precede it.  As above, suppose
 that $e = e_1 + \cdots + e_m$ and $1 - e = e_{m+1} + \cdots + e_n$, where the $e_i$ are primitive.

Regarding condition (i) of the claim:  By strongness of $_{\la} T$, 
the functor $\Hom_{\la}( - , T):  \lamod \rightarrow \modlatilde$ 
induces a contravariant equivalence between $\pinf (\lamod)$
 and a
 certain subcategory
 of $\pinf(\modlatilde)$, whence $\pdim_{\latilde} \Hom_{\la}(S_i, T) < \infty$ 
for all $i \ge m
+1$ (see \cite [Reference Theorem III and Theorem 1]{Birgedualities} or \cite [Theorems 7,8]{HuisgenZimmermann-Saorin}).  So we only need to show that  the right $
\latilde$-module $\Hom_{\la}(S_i, T)$ is isomorphic
 to $\etilde_i \latilde/ \etilde_i \Jtilde$ for $m+1 \le i \le n$.  Let $i \ge 
m+1$.  In light of Proposition \ref{prop.decomposition-strongtilting}, $S_i = \Delta(\soc T_i)$ is the only occurrence of $S_i$ in the socle of $T$, and 
therefore any homomorphism $S_i \rightarrow T$ maps $S_i$ onto $\Delta(\soc T_i)$. If $\text{in}_i: S_i 
\hookrightarrow T_i \subseteq T$ is an embedding, we thus find that $\Hom_{\la}(S_i, T) \cong \text{in}_i\, \latilde$.  
One checks that the latter module is annihilated by the radical $\Jtilde$ of $\latilde$, but not annihilated by $\etilde_i
$, and concludes that 
$\Hom_{\la}(S_i, T) \cong \etilde_i \latilde/ \etilde_i \Jtilde$ as postulated.  

To verify condition (ii), namely finiteness of $\pdim_{\etilde \latilde \etilde}\, (1 - \etilde) \latilde \etilde$,
 we again use Proposition \ref{prop.decomposition-strongtilting}. The isomorphism $\etilde \latilde \etilde \cong \End_{\la}(T')^{\op}$ 
shows that our claim 
amounts to finite projective dimension of $\Hom_{\la} (T'', T')$ over $\Gamma = \End_{\la}(T')^{\op}$.   
Since $\add(T'')\linebreak 
\subseteq \add\bigl(\bigoplus_{m+1 \le i \le n} \mathcal{A}_i\bigr)$
 by the  proposition, it suffices to show 
that $\pdim _{\Gamma} \, \Hom_{\la} (\mathcal{A}_i, T')\linebreak < \infty$  for $i \ge m+1$.
  Fix $i \ge m+1$ in the following. Towards another reduction 
step, we apply Proposition \ref{prop.explicit calculation of approximation} to a minimal $\pinf (\lamod)$-approximation
 $\widehat{\mathcal{A}}_i$ of $E(S_i)_\sigma$. Part (c) of 
 \ref{prop.explicit calculation of approximation} provides us with
 a map $\rho:  \mathcal{A}_i \rightarrow \widehat{\mathcal{A}}_i$ such that
 both  $\Ker(\rho)$ and $\Coker(\rho)$ 
belong to $\mathcal{T}$. 
 Since the restriction of the functor $\Hom_{\la}(-, T)$ to $\pinf(\lamod)$ is exact, so 
is $\Hom_{\la}(-, T')_{|_{\pinf(\lamod)}}$, and consequently torsionfreeness of $T'$ 
shows $\Hom_{\la}(\rho, T'): \Hom_{\la}(\widehat{\mathcal{A}}_i, 
T') \rightarrow \Hom_{\la}(\mathcal{A}_i, T')$ 
to be an isomorphism of right $\Gamma$-modules.  Lemma \ref{lem.preserving-reflecting finprojdim} moreover ensures 
that $e \widehat{\mathcal{A}}_i$ has finite 
projective dimension in $\elaemod$.  Returning to the category 
equivalence $\mathcal{G} \cap (\lamod) \approx \elaemod$ 
which 
 sends $M$ to $eM$ (cf.  Lemma \ref{2.2}), 
we thus find finiteness of $\pdim_{\Gamma} \Hom_{\la}(\widehat{\mathcal{A}}_i, T')$ 
to be equivalent to finiteness of the 
projective dimension of
 $\Hom_{e \Lambda e} (e \widehat{\mathcal{A}}_i, e T')$ over $\End_{e \Lambda e}( e T')^{\op}$ alias $\etilde \latilde 
\etilde$.  Given that  $\Hom_{e \Lambda e}(- ,  e T')$ takes objects in $\pinf(\elaemod)$ to objects in $\pinf(\modetillatiletil)$ [loc.cit., applied to $e \Lambda e$ in place of $\la$], condition (ii) follows.
\end{proof}

Theorem \ref{Birge5.3} ensures  that, with a ``duplicate" in $\latilde$ of the original idempotent $e \in \la$, the test provided by Theorem \ref{thm.reduction-to-simples-of-infteprojdim} is again available towards deciding whether $\pinf(\modlatilde)$ is contravariantly finite in $\modlatilde$. 

\begin{cor} \label{Birge5.4}
 Adopt the hypotheses and notation of Theorem \ref{Birge5.3}.  Then contravariant finiteness of $\pinf(\modlatilde)$ in $\modlatilde$ is equivalent to
  contravariant finiteness of $\pinf(\modetillatiletil)$ in\ $\modetillatiletil$.  \end{cor}

In light of the fact that the basic strong tilt $\widetilde{e \Lambda e}$ of $e \Lambda e$ coincides with
 $\etilde \latilde \etilde$,  
we thus obtain:  Not only is existence of a strong tilting object 
in
 $\lamod$ 
equivalent to existence of a strong tilting object in $\elaemod$ under  conditions \ref{condition4.1},
 but the same hypothesis implies that  $\lamod$
 allows for arbitrary repetitions of strong tilting precisely when this is true for $e \Lambda e$-mod.  This conclusion compiles information from Theorems \ref{teor.strong-tilting-iteration}, \ref{thm.reduction-to-simples-of-infteprojdim} and Corollary \ref{Birge5.4}.  

\begin{cor} \label{cor.iteration-stongtilting} 
 Let $\la$ be a basic Artin algebra and $e \in \la$ an idempotent 
satisfying the hypotheses $4.1$ $($i.e., all simple left $\la$-modules of infinite projective dimension belong
 to 
$\add(\Lambda e/ J e)$,
 and $\pdim_{e \Lambda e} e \Lambda (1- e) < \infty$$)$.  
Then the category $\lamod$ allows 
for unlimited iteration of strong tilting if and only if  the same is true for the category $\elaemod$.
\end{cor}

  A first straightforward application of our techniques shows that, in testing for contravariant finiteness of $\pinf(\lamod)$ or iterability of strong tilting, we may automatically discard the idempotents corresponding to the simple left $\la$-modules of projective dimension at most $1$. Namely: 

\begin{prop} \label{prop.elimination-simples-projdim1}
Let $\Lambda$ be a basic Artin algebra on which we fix the complete set of primitive idempotents $\{e_1,...,e_r,e_{r+1},...,e_n\}$, ordered in such a way that $\pdim(\Lambda e_i/Je_i)\leq 1$ for $i>r$. For $e =\sum_{i=1}^re_i$, the following assertions are equivalent: 

\begin{enumerate}
\item $\mathcal{P}^{<\infty}(\lamod)$ is contravariantly finite in $\lamod$ $($resp., $\lamod$ allows for unlimited iteration of strong tilting$)$;
\item $\mathcal{P}^{<\infty}(\elaemod)$ is contravariantly finite in $\elaemod$ $($resp., $\elaemod$ allows for unlimited iteration of strong tilting$)$.
\end{enumerate}
\end{prop}
\begin{proof}
From the proof of \cite[Proposition 2.1]{FS}, we know that $e\Lambda (1-e)$ is projective as a left $e\Lambda e$-module. Thus the pair $(\lamod , e)$ satisfies the blanket hypotheses of  Setting \ref{set.initial setting}.  Assertions 1 and 2 are thus consequences of Theorem \ref{thm.reduction-to-simples-of-infteprojdim} and Corollary \ref{cor.iteration-stongtilting}, respectively. 
\end{proof}

\section { Applications and examples}
By way of the equivalences established in the previous sections, we can now easily secure contravariant finiteness of $\pinf (\lamod) $ and 
$\pinf(\modlatilde)$ in cases in which this originally required a considerable effort.  
Moreover, the unifying reasons behind these results become more transparent through the reduction and permit us to expand the settings to which they apply.

\subsection{Precyclic/postcyclic vertices and normed Loewy lengths}

An instance in which the simplification gained by reduction to corner algebras stands out is that of truncated path algebras and their strong tilts (see \cite{Birgedualities} and \cite{HuisgenZimmermann-Saorin}); without a reduction technique, it is challenging to confirm that $\lamod$ allows for unlimited iteration of strong tilting in this case. 

In light of Theorem \ref{thm.reduction-to-simples-of-infteprojdim}, the first step, namely to confirm contravariant finiteness of 
$\pinf (\lamod)$ for truncated $\la$, has now been trivialized.  To generalize it, recall that, given any path algebra modulo relations, $\la = KQ/I$, a vertex $e_i$ of $Q$ (systematically identified with a primitive idempotent of $\la$) is called {\it precyclic\/} if there exists a path of length $\ge 0$ which starts in $e_i$ and ends on an oriented cycle; the attribute {\it postcyclic\/} is dual, and $e_i$ is called {\it critical\/} if it is both pre- and postcyclic.  Clearly, all vertices which give rise to simple modules of infinite projective dimension are among the precyclic ones; moreover $e_i \Lambda e_j =  0$ whenever $e_i$ is precyclic, but $e_j$ is not.  Theorem 4.6 thus yields

\begin{prop} \label{prop.precyclic-contrfiniteness}
Let $\la = KQ/I$ be an arbitrary path algebra modulo relations,and let $e$ be the sum of the primitive idempotents corresponding to the precyclic vertices of $Q$.  
If $\pinf(\elaemod)$ is contravariantly finite in $\elaemod$, e.g., if 
$\lfindim e \Lambda e = 0$, then $\pinf (\lamod)$ is contravariantly finite in $\lamod$.
\end{prop}

For truncated $\la$, the final condition concerning the left finitistic dimension of $e \Lambda e$ is clearly satisfied since all indecomposable projective left $e \Lambda e$-modules $e \Lambda e_i$ for precyclic $e_i$ have the same Loewy length; this is not necessarily true for the indecomposable injective left $e \Lambda e$-modules, but it is for those whose socles correspond to critical vertices; namely, if $e'$ is the sum of the critical vertices of $Q$, then the indecomposable injective left $e' \Lambda e'$-modules also have coinciding Loewy lengths in the truncated case. The combination of these two conditions which norm the  Loewy lengths of certain projective or injective modules is, in fact, all that is needed to guarantee that $\lamod$ allows for iterated strong tilting. 

\begin{prop} \label{prop.precyclic-stiltingiteration}
Again, let $\la = KQ / I$ be a path algebra modulo relations and $e$ the sum of the primitive idempotents corresponding to the precyclic vertices of $Q$.  Moreover, let $e'$ be the sum of those idempotents which correspond to the critical vertices.  \smallskip 

Suppose that all indecomposable projective left $e \Lambda e$-modules have the same Loewy length, and that the analogous equality holds for the Loewy lengths of the indecomposable injective left $e' \Lambda e'$-modules.  Then $\lamod$ allows for unlimited iteration of strong tilting, thus giving rise to a sequence of related module categories 
$\lamod\, \rightsquigarrow\, \modlatilde\, \rightsquigarrow\, \latiltilmod\, \rightsquigarrow \cdots$.  
Starting with the first strong tilt, 
$\modlatilde$, 
the Morita equivalence classes of these categories repeat periodically with period $2$.
\end{prop}
\begin{proof}
Suppose that $e =  e_1 + \cdots + e_m$ and $e' =  e_1 + \cdots + e_r$ for some $r \le m$.  Since
 $\lfindim e \Lambda e = 0$ due to the first condition on Loewy lengths, $\pinf(\elaemod)$ is contravariantly finite in $\elaemod$.  Proposition \ref{prop.precyclic-contrfiniteness} thus guarantees that $\pinf (\lamod)$ is contravariantly finite in $\lamod$.  

As we already pointed out above, the conditions
(i) and (ii) of Setting 4.1 are satisfied for the pair $(\lamod, e)$, whence, by Theorem \ref{Birge5.3}, we only
need to show that $e \Lambda  e\text{-mod}$ allows for unlimited iteration of strong tilting.  Set $\la' = e \Lambda e$, and let 
$Q'$, resp. $J'$, be the quiver and Jacobson radical of $\la'$, respectively.  We already know that the category 
$\pinf(\laprimemod)$ is contravariantly finite in $\laprimemod$, due to the vanishing of the left finitistic 
dimension of $\la'$.  The latter in fact entails that the basic strong left $\la'$-module $T'$ is a copy of the left 
regular module $_{\la'} \la'$.  Consequently, the strongly tilted algebra $\widetilde{\la'} = \widetilde{e \Lambda e}$, i.e., 
the opposite of $\End_{\la'}(T')$, coincides with $\la'$.  By Theorem \ref{thm.reduction-to-simples-of-infteprojdim}, it therefore suffices to check that $
\pinf(\modla')$ is contravariantly finite in $\modla'$.  To confirm this, we observe that the pair $
(\modla', e')$ in turn satisfies the hypotheses 4.1:  Indeed, the precyclic vertices of $(Q')^{\op}$ are 
precisely $e_1, \dots, e_r$, whence the simple right $\la'$-modules of infinite projective dimension are among 
the quotients $e_i \la' /e_i J'$ for $i \le r$, and $(e - e') \la' e' = 0$.  The right finitistic dimension of $e' \la' e' = e' 
\Lambda e'$ is in turn zero, because all indecomposable projective right $e' \Lambda e'$-modules have the same Loewy 
length; indeed, this follows by duality from the second of our two hypotheses.  Consequently, another 
application of Proposition \ref{prop.precyclic-contrfiniteness} yields contravariant finiteness of $\pinf(\modla')$ as required.

The concluding statements are part of Theorem \ref{teor.strong-tilting-iteration}. 
\end{proof}

We deduce Theorem D of \cite{Birgedualities} as a special case.  
\begin{cor} \label{cor.Birge-Theorem D}
Suppose that $\la = KQ/I$ is a truncated path algebra.  
Then $\lamod$ allows for unlimited iteration of strong tilting. 
\end{cor}

The proof of the following generalization of Proposition \ref{prop.precyclic-stiltingiteration} is immediate from that of the latter.

\begin{cor}  \label{cor.precyclic-stilting iteration}
Let $\la = KQ/I$, and let $e$ and $e'$ be as in the statement of \ref{prop.precyclic-stiltingiteration}. 
 If $\,\lfindim e\la e = 0$
 and $\pinf(\modeprimelaeprime)$ is contravariantly finite in $\modeprimelaeprime$, then $\lamod$ allows for unlimited iteration of strong tilting. 
\end{cor}
Recall that an Artin algebra $\la$ is said to be {\it left serial\/} if all indecomposable projective left $\la$-modules are uniserial.  The algebras which are left and right serial are also called {\it Nakayama algebras\/}.  In \cite{BH} it was shown that, for any split left serial algebra $\la$, the category $\pinf (\lamod)$ is contravariantly finite in $\lamod$.  The conclusion actually carries over to the category $\modla$ of right $\la$-modules.  Namely:

\begin{prop} \label{prop.serial algebras}
 Suppose that $\la$ is a path algebra modulo relations.  If $\la$ is left serial, then 
$\pinf (\lamod)$ is contravariantly finite in $\lamod$ and $\pinf(\modla)$ is contravariantly finite in $\modla$.
\end{prop}
\begin{proof}  The assertion for left modules was proved in \cite{BH}.  We only address right $\la$-modules.

By hypothesis, $\la \cong KQ/I$ is left serial.  In particular, this means that $Q$ is free of double arrows.  Without loss of generality, we assume that $Q$ is a connected quiver; we may further assume that it is not a tree, since otherwise $\la$ has finite global dimension, which renders the contravariant finiteness claim trivial.  By left seriality, $Q$ then contains a single oriented cycle such that all off-cycle vertices are pre- but not postcyclic.  Consequently, any vertex of $Q^{\op}$ either belongs to said cycle or else is  post- but not precyclic.  On letting $e$  be the sum of the precyclic vertices of $Q^{\op}$, i.e., the vertices located on the oriented cycle in the present situation, we thus obtain a Nakayama algebra $e \Lambda e$.  Hence the $\pinf$-categories in both $\elaemod$ and $\modelae$ are contravariantly finite in the corresponding ambient module categories by \cite[Theorem 5.2]{BH}.  (The latter may alternatively be deduced from the fact that Nakayama algebras have finite representation type.) On combining this with Proposition \ref{prop.precyclic-contrfiniteness}, we obtain the claim.  
\end{proof}

\subsection{Applications to Morita contexts}

 In this  subsection we shall see that Morita contexts provide a tool to construct  examples of $\mathcal{P}^{<\infty}$-contravariant finiteness and iteration of strong tilting.   Recall that  a \emph{Morita context} (over a ground commutative ring $K$) consists of a sextuple $(A,B,M,N,\varphi ,\psi)$, where $A$ and $B$ are $K$-algebras, $M$ and $N$ are an $A$-$B$-  and a  $B$-$A$-bimodule, which we always assume with the same action of $K$ on the left and on the right,  and $\varphi :M\otimes_BN\longrightarrow A$ and  $\psi :N\otimes_AM\longrightarrow B$ are morphisms of $A$-$A$- and $B$-$B$-bimodules, respectively, satisfying certain compatibiliity conditions (see \cite{Muller}) which are exactly the ones that make $\Lambda =\begin{pmatrix}A & M\\ N & B \end{pmatrix}$ into a $K$-algebra with the obvious multiplication. Recall that $\tau_A=\text{Im}(\varphi)$ and  $\tau_B=\text{Im}(\psi)$ are two-sided  ideals of $A$ and $B$, respectively, called the \emph{trace ideals} of the Morita context. 
 
The Morita contexts in which we are interested  have some additional properties. We assume that $K$ is  artinian and $A$, $B$,  $M$ and $N$ are finitely generated as $K$-modules. Such a Morita context will be called a \emph{basic Morita context of Artin algebras} if $\Lambda$ is basic or, equivalently, if $A$ and $B$ are basic, $\tau_A\subseteq J(A)$ and $\tau_B\subseteq J(B)$, where $J(-)$ denotes the Jacobson radical.

 The following is the main result of the subsection.

 \begin{teor} \label{thm.mainMoritacontexts}
  Let $(A,B,M,N,\varphi ,\psi)$ be a basic Morita context of Artin algebras, where $M$ is   projective as a left $A$-module. Suppose also that $\psi:N\otimes_AM\longrightarrow B$ is a monomorphism and the algebra $B/\tau_B$ has finite global dimension. 
  
  Set $\Lambda =\begin{pmatrix} A & M\\ N & B\end{pmatrix}$. Then $\mathcal{P}^{<\infty}(\lamod)$ is contravariantly finite in $\lamod$ (resp.,  $\lamod$ allows for unlimited iteration of strong tilting) if and only if this is true for $\pinf(\Amod)$ (resp., $\Amod$).
  
  In particular, $\lamod$ allows for unlimited iteration of strong tilting when $A$ is a Gorenstein algebra or a truncated path algebra or when $\lfindim A=0=\rfindim A$.
 \end{teor}
 \begin{proof}
 Take $e=\begin{pmatrix}1 & 0\\ 0 & 0 \end{pmatrix}$. The goal is to check that the pair $(\lamod ,e)$  satisfies the blanket hypotheses of Setting \ref{set.initial setting}. Bearing in mind that the algebra $e\Lambda e$ is isomorphic to $A$, the result will  then be a direct consequence of Theorem \ref{thm.reduction-to-simples-of-infteprojdim} and  Corollary \ref{cor.iteration-stongtilting}. Note that $B$ and $(1-e)\Lambda (1-e)$ are also isomorphic algebras and hence the $e\Lambda e$-$(1-e)\Lambda (1-e)$-bimodule $e\Lambda (1-e)$ is isomorphic to the $A$-$B$-bimodule $M$, with the obvious meaning. Hence $\pdim_{e\Lambda e}(e\Lambda (1-e))=0$ and we need to check that $\pdim_\Lambda (\frac{\Lambda (1-e)}{J(\Lambda))(1-e)})<\infty$. 
 
 We first prove that $\Lambda e\Lambda (1-e)=\begin{pmatrix}0 & M\\ 0 & \tau_B \end{pmatrix}\cong\begin{pmatrix}M\\ \tau_B \end{pmatrix}$ is a projective left $\Lambda$-module. Indeed the adjunction map $\epsilon_{\Lambda (1-e)}:\Lambda e\otimes_{e\Lambda e} e\Lambda (1-e)\longrightarrow\Lambda (1-e)$ gets identified, in the obvious way, with the map $\begin{pmatrix} A\\ N \end{pmatrix}\otimes_AM\longrightarrow\begin{pmatrix}  M\\ B\end{pmatrix}$, that can be expressed matricially as $\begin{pmatrix} \mu & 0\\ 0 & \psi \end{pmatrix}:\begin{pmatrix}A\otimes_AM\\ N\otimes_AM \end{pmatrix}\longrightarrow\begin{pmatrix}M\\ B \end{pmatrix}$. Here $\mu:A\otimes_AM\longrightarrow M$ is the canonical isomorphism given by multiplication and $\psi:N\otimes_AM\longrightarrow B$ is the map in the Morita context. It follows that $\epsilon_{\Lambda (1-e)}$ is a monomorphism, and so 
 $$\Lambda e\otimes_{e\Lambda e}e\Lambda (1-e)\cong\Lambda e\Lambda (1-e)=\text{Im}(\epsilon_{\Lambda (1-e)}).$$
 But $\Lambda e\otimes_{e\Lambda e}e\Lambda (1-e)\in\text{add}(\Lambda e\otimes_{e\Lambda e}e\Lambda e)=\text{add}(\Lambda e)$ since $e\Lambda (1-e)\in e\Lambda e\text{-proj}=\text{add}(_{e\Lambda e}e\Lambda e)$. Therefore  $\Lambda e\Lambda (1-e)$ is a projective left $\Lambda$-module.

 Next we should notice that there is an algebra isomorphism $\frac{\Lambda}{\Lambda e\Lambda}\cong\frac{B}{\tau_B}$, so that  these isomorphic algebras have finite global dimension. Moreover, we have an isomorphism $\frac{\Lambda}{\Lambda e\Lambda}\cong\frac{\Lambda (1-e)}{\Lambda e\Lambda (1-e)}$ in $\lamod$, which, by the previous paragraph,  implies that $\frac{\Lambda}{\Lambda e\Lambda}$ has projective dimension $\leq 1$ as a left $\Lambda$-module.  Note also that $\frac{\Lambda (1-e)}{J(\Lambda)(1-e)}$ is canonically a left $\frac{\Lambda}{\Lambda e\Lambda}$-module.
 Fix a minimal projective resolution $0\rightarrow Q_t\rightarrow ...\rightarrow Q_1\rightarrow Q_0\rightarrow \frac{\Lambda (1-e)}{J(\Lambda) (1-e)}\rightarrow 0$ in $\frac{\Lambda}{\Lambda e\Lambda}\text{-mod}$. This is also an exact sequence in $\lamod$, and we have that $\pdim_\Lambda (Q_k)\leq 1$ since $Q_k\in\text{add}(\frac{\Lambda}{\Lambda e\Lambda})$, for all $k=0,1,...,t$. It then follows that $\pdim_\Lambda (\frac{\Lambda (1-e)}{J(\Lambda)(1-e)})<\infty$, as desired. 
 
 The final statement clearly follows from assertion 3 since in those examples $A\text{-mod}$ allows for unlimited iteration of strong tilting. 
 \end{proof}

The last theorem and its proof have the following consequence for  triangular matrix algebras.
 
  \begin{cor} \label{cor.triangular matrix}
 Let $K$ be a  commutative Artinian ring, let $A$ and $B$ be basic Artin $K$-algebras and let $M$ and $N$ be a finitely generated $A$-$B$-bimodule and $B$-$A$-bimodule, respectively. Suppose that $B$ has finite global dimension and  $\mathcal{P}^{\infty}(\Amod)$ is contravariantly finite in $\Amod$. The following assertions hold:
 
 \begin{enumerate}
 \item If $\Lambda =\begin{pmatrix} A & 0\\ N & B \end{pmatrix}$, then  $\mathcal{P}^{\infty}(\lamod)$ is contravariantly finite in $\lamod$. Moreover if $\Amod$ allows for unlimited iteration of strong tilting, so does $\lamod$.
 \item If $\pdim(_AM)<\infty$, the conclusions of assertion 1 remain true on replacement of $\Lambda$ by  $\Gamma =\begin{pmatrix} A & M\\ 0 & B \end{pmatrix}$.
 \end{enumerate}
 \end{cor}
 \begin{proof}
 Assertion 1 is a direct consequence of Theorem \ref{thm.mainMoritacontexts}.  As for assertion 2, note that if  $e=\begin{pmatrix}1 & 0\\ 0 & 0 \end{pmatrix}$ then  $\pdim_{e\Gamma e}(e\Gamma (1-e))<\infty$ since we have $\pdim(_AM)<\infty$. We will prove that $\pdim_\Gamma (\Gamma e\Gamma (1-e))<\infty$ and, arguing as in the last paragraph of the proof of  Theorem \ref{thm.mainMoritacontexts}, we will conclude that the pair $(\Gamma\text{-mod}, e)$ satisfies the blanket hypotheses of Setting  \ref{set.initial setting}, and the result will follow from Theorem \ref{thm.reduction-to-simples-of-infteprojdim} and  Corollary \ref{cor.iteration-stongtilting}.
 
 %Our goal is to check that the pair $(\Gamma ,e)$ satisfies the blanket hypotheses of Setting  \ref{set.initial setting}, and the proof will be finished arguing as in the proof of the mentioned theorem. Since the bimodule  $_{e\Gamma e}e\Gamma (1-e)_{(1-e)\Gamma (1-e)}$ is canonically identified with the bimodule $_AM_B$, we have that $\pdim(_{e\Gamma e}e\Gamma (1-e))<\infty$ and only need to check that $\pdim_\Gamma (\frac{\Gamma (1-e)}{J(\Gamma)(1-e)})<\infty$. 
 
  Note that  $0=(1-e)\Gamma e$ and so $\Gamma e=e\Gamma e$ is projective as right $e\Gamma e$-module. Moreover  we have a commutative diagram of $K$-modules
$$\xymatrixrowsep{2.5pc}\xymatrixcolsep{4pc}\xymatrix{ e\Gamma e \otimes _ {e\Gamma e} e\Gamma (1-e) \ar[r]^-{\cong} \ar@{=}[d]    & e\Gamma (1-e)   \ar@{=}[d]   \\
\Gamma e \otimes _ {e\Gamma e} e\Gamma (1-e) \ar[r]  & \Gamma e  \Gamma (1-e) 
},$$
 where the upper horizontal arrow is the canonical isomorphism and the lower horizontal one is the multiplication map. It follows that this latter arrow is an isomorphism, which implies that  $\pdim(_\Gamma\Gamma e\Gamma (1-e))<\infty$ since $_{e\Gamma e}e\Gamma (1-e)$ has finite projective dimension and the functor $\Gamma e\otimes_{e\Gamma e}-:e\Gamma e\text{-mod}\longrightarrow\Gamma\text{-mod}$ is exact and takes projectives to projectives. 
 \end{proof}

% A particular case of interest in the situation of last proposition is when $N\otimes_AM=0$. The following result gives a handling characterization of this property when $_AM$ is projective.
 
% \begin{lema} \label{lem.N-tensor-M equals 0}
 %Let  $(A,B,M,N,\varphi ,\psi)$ be a basic Morita context of Artin $K$-algebras such that $_AM$ is projective. The following assertions are equivalent:
 
 %\begin{enumerate}
% \item $N\otimes_AM=0$;
% \item  $Ne_i=0$ whenever $e_i$ is a primitive idempotent of $A$ such that $Ae_i/J(A)e_i$ embeds in the top of $_AM$;
% \item  $\psi = 0$ and   $\begin{pmatrix} M\\ 0 \end{pmatrix}$ is projective as a left $\la$-module.  
 %\end{enumerate}
 %\end{lema}
% \begin{proof}
% $(1)\Longleftrightarrow (3)$ is a particular case of the equivalence $(1)\Longleftrightarrow (2)$ in Proposition \ref{prop.crucialMoritacontextresult}.
 
% $(1)\Longleftrightarrow (2)$ Let $I=\{1,...,r\}$ be as in the proof of Proposition \ref{prop.crucialMoritacontextresult}, so that we have an isomorphism $M\cong\oplus_{i=1}^r(Ae_i)^{t_i}$. It follows that $N\otimes_AM=0$ if, and only if, $N\otimes_AAe_i=0$ for $i=1,..,r$. This is clearly  equivalent to the equality $Ne_i=0$, for all  $i=1,..,r$.
 %\end{proof}
 
 We end the paper by giving non-triangular examples to which Theorem \ref{thm.mainMoritacontexts} applies. We start with the following elementary observation. 
 
  \begin{rem} \label{rem.effective construction} 
  Let $A$ and $B$ be basic Artin $K$-algebras and $M$ and $N$  finitely generated $A$-$B$- and $B$-$A$-bimodules, respectively. If $N\otimes_BM=0$ then any morphism  $\varphi :M\otimes_BN\longrightarrow A$ of $A$-$A$-bimodules such that $\Im(\varphi)\subseteq J(A)$ gives rise to a basic Morita context of Artin algebras $(A,B,M,N,\varphi ,0)$ since the required  compatibility conditions hold (see, e.g., \cite[Exercise IV.35]{STEN}). 
  \end{rem}

   \begin{ejems} \label{ejems.effective construction} 
   Let $A$, $B$ be basic finite dimensional algebras over an algebraically closed field $K$ and let  us fix complete sets of primitive idempontents $\{e_1,....,e_m\}$ and  $\{e_{m+1},....,e_n\}$ in $A$ and $B$, respectively. Consider either one of the following two situations, where the unadorned $\otimes$ means $\otimes_K$:
   
   \begin{enumerate}
   \item[(a)] $\gldim(B)<\infty$, $M$ is any finitely generated $A$-$B$-bimodule that is projective as a left $A$-module, $N$ is any finitely generated $B$-$A$-bimodule such that $Ne_i=0$ whenever the simple left $A$-module $Ae_i/J(A)e_i$ embeds in $\top(_AM)$, and   $\varphi :M\otimes_BN\longrightarrow A$ is any morphism of $A$-$A$-bimodules such that $\text{Im}(\varphi)\subseteq J(A)$. Then   $(A,B,M,N,\varphi ,0)$  is a basic Morita context.

   \item[(b)] Suppose  that the quiver $Q_B$ of $B$ has no oriented cycles, let $\Soc(B)$ be the socle of $B$ as a $B$-$B$-bimodule and fix any index $k\in\{m+1,...,n\}$ such that $W:=\Soc(B)e_k\neq 0$. Note that $W$ is a subbimodule of $_BB_B$ isomorphic to $Y\otimes\frac{e_kB}{e_kJ(B)}$, for some semisimple left $B$-module $Y$ that, due to the absence of oriented cycles in $Q_B$, satisfies  $e_kY=0$. Take $M=Ae_i\otimes\frac{e_kB}{e_kJ(B)}$ and $N=Y\otimes\frac{e_iA}{e_iJ(A)}$, for some $i=1,...,m$. We then have  $M\otimes_BN=0$ and an isomorphism of $B$-$B$-bimodules 
   $$N\otimes_AM\cong  Y\otimes\frac{e_iAe_i}{e_iJ(A)e_i}\otimes \frac{e_kB}{e_kJ(B)}\cong Y\otimes\frac{e_kB}{e_kJ(B)}\cong W.$$ 
   Taking as $\psi :N\otimes_AM\longrightarrow B$ the composition of this latter isomorphism followed by the inclusion $W\hookrightarrow B$, we get a basic Morita context $(A,B,M,N,0,\psi )$ (cf. the left-right symmetric version of Remark \ref{rem.effective construction}).
   \end{enumerate}
   
   In the situations (a) and (b), the associated algebra $\Lambda =\begin{pmatrix} A & M\\ N & B\end{pmatrix}$ satisfies that $\mathcal{P}^{<\infty}(\lamod)$ is contravariantly finite in $\lamod$ $($resp.,  $\lamod$ allows for unlimited iteration of strong tilting$)$ if, and only if, so does the algebra $A$.
   \end{ejems}
   \begin{proof}
  In situation (a), we have an isomorphism of left $A$-modules $M\cong\oplus_{i\in I}Ae_i^{t_i}$, where $I$ is the set of $i\in\{1,...,m\}$ such that the simple left $A$-module $Ae_i/J(A)e_i$  embeds in $\text{top}(_AM)$ and $t_i>0$ for all $i\in I$. It follows that $N\otimes_AM=0$ and the existence of the mentioned Morita context follows by Remark \ref{rem.effective construction}.   Hence in both situations $M$ is projective as a left $A$-module and the map $\psi :N\otimes_AM\longrightarrow B$ is a monomorphism.  In (a) we have that $\tau_B=0$ and in (b)  the quiver of $B/\tau_B$ has no oriented cycles. Therefore  $\text{gl.dim}(B/\tau_B)<\infty$ in both cases and assertions 1 and 2 are direct consequences of Theorem \ref{thm.mainMoritacontexts}. 
   \end{proof}
   
The following example  is a combinatorial version of Example \ref{ejems.effective construction}(a).

\begin{ejem} \label{ex.merging-of-quivers}
Let $A:=KQ_e/\langle R_e \rangle$ and $B:=KQ_{1-e}/\langle R_{1-e} \rangle$ be finite dimensional algebras given as quotients of path algebras modulo relations, on which we fix bases  $B_e$ and $B_{1-e}$ consisting of paths in $Q_e$ and $Q_{1-e}$, respectively. Consider the algebra $\Lambda =KQ/\langle R \rangle$, also given by quiver and relations, where:
\begin{enumerate}
\item The quiver $Q$ is obtained   from the disjoint union quiver $Q_e\bigsqcup Q_{1-e}$ by adding two  finite sets of arrows (some possibly empty) $\{\alpha_1,...,\alpha_s\}$  and $\{\beta_1,...,\beta_t\}$, where the $\alpha_i$ go from vertices in $Q_{1-e}$ to vertices in $Q_{e}$ and the $\beta_j$ from vertices in $Q_e$ to vertices in $Q_{1-e}$, with the only restriction that $t(\alpha_k)\neq o(\beta_l)$ for all $k=1,...,s$ and $l=1,...,t$ (here $o(\gamma)$ and $t(\gamma )$ denote, respectively,  the origin and terminus of any arrow $\gamma$); 
\item The set of relations is $R=R_e\cup R_{1-e}\cup R_{1-e,e}\cup R_{e,e}$, where
\begin{center}
$R_{1-e,e}=\{\beta_lq\text{: }l\in\{1,...,t\}\text{ and }q\in\bigcup_{k=1}^se_{o(\beta_l)}B_ee_{t(\alpha_k)}\}$, and 

$R_{e,e}=\{\alpha_kp\beta_l\text{: }l\in\{1,...,t\}\text{, }k\in\{1,...,s\}\text{ and }p\in e_{o(\alpha_k)}B_{1-e}e_{t(\beta_l)}\}$.
\end{center}

\end{enumerate}

When $\gldim(B)<\infty$, the subcategory
 $\mathcal{P}^{<\infty}(\lamod)$ is contravariantly finite in $\lamod$
(resp. $\lamod$ allows for arbitrary iteration of strong tilting) if, and only if, the corresponding property is true for the algebra $A$.
\end{ejem}

\begin{proof}
Let $\{e_1,...,e_m, e_{m+1},...,e_n\}$ be the set of primitive idempotents of $\Lambda$ corresponding to the vertices of $Q$, where we assume that $i$ is a vertex of $Q_e$ if and only if $i\leq m$. We have the canonical Morita context associated to $e:=e_1+...+e_m$, so that $\Lambda \cong \begin{pmatrix} e\Lambda e & e\Lambda (1-e)\\ (1-e)\Lambda e & (1-e)\Lambda (1-e) \end{pmatrix}$. The relations $R_{1-e,e}$ and $R_{e,e}$  guarantee that the nonzero paths in $e\Lambda e$ and $(1-e)\Lambda (1-e)$ consist exclusively of arrows of $Q_e$ and $Q_{1-e}$, respectively. That is, we have $e\Lambda e\cong A$ and $(1-e)\Lambda (1-e)\cong B$.

On the other hand, one can check that the chosen set of relations implies that $e\Lambda (1-e)=\sum_{k=1}^s\sum_{p\in e_{o(\alpha_k)}B_{1-e}}e\Lambda\alpha_kp$, that this sum is direct and that the map $e\Lambda e_{t(\alpha_k)}\longrightarrow e\Lambda\alpha_kp$ ($x\rightarrow x\alpha_kp$) is an isomorphism, for all $k=1,...,s$. Therefore $e\Lambda (1-e)$ is a projective left $e\Lambda e$-module. Moreover the simple left $e\Lambda e$-module ${e\Lambda e_i}/ {eJe_i}$ embeds in $\top (_{e\Lambda e}e\Lambda (1-e))$ if and only if $i=t(\alpha_k)$ for some $k=1,...,s$. But the relations in $R_{1-e,e}$ imply that $(1-e)\Lambda e_{t(\alpha_k)}=0$, for all $k=1,...,s$. 

We are thus in the situation of Example \ref{ejems.effective construction}(a), and the conclusions follow from that example.
\end{proof}

\subsection{A specific path algebra modulo relations}  The final example is a non-monomial path algebra modulo relations whose category of left modules allows for unlimited iteration of strong tilting. Our reduction technique renders verification of this fact significantly less labor-intensive.

\begin{ejem}
Let $\Lambda  = KQ/I$ be a specimen of the following class of finite dimensional algebras over a field $K$, which depends on $4$ parameters $c_1, \dots, c_4 \in K^*$.  The quiver $Q$ is
$$\xymatrixcolsep{4pc}
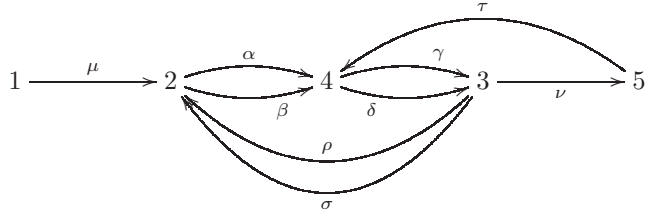
\xymatrix{
1 \ar[r]^{\mu} &2 \ar@/^/[r]^{\alpha} \ar@/_/[r]_(0.7){\beta} &4 \ar@/^/[r]^(0.7){\gamma} \ar@/_/[r]_(0.3){\delta} &3 \ar@/^2.5pc/[ll]_{\rho} \ar@/^3.5pc/[ll]^{\sigma} \ar[r]_{\nu} &5 \ar@/_2pc/[ll]_{\tau}
}$$
\noindent and $I \subseteq KQ$ is the ideal generated by $\gamma \alpha - c_1 \delta \beta$, $\gamma \beta - c_2 \delta \alpha$, $\alpha \rho - c_3 \beta \sigma$, and $\alpha \rho - c_4 \tau \nu$,  next to monomial relations which are apparent from the graphs of the indecomposable projective left $\Lambda$-modules:
$$\xymatrixrowsep{1.5pc}\xymatrixcolsep{1pc}
\xymatrix{
&1 \edge[d]^{\mu} &&&&2 \edge[dl]_{\alpha} \edge[dr]^{\beta} &&&&3 \edge[dl]_{\rho} \edge[dr]_{\sigma} \edge[drr]^{\nu}  \\
&2 \edge[dl]_{\alpha} \edge[dr]^{\beta} &&&4 \edge[d]_{\gamma} \edge[drr]^(0.2){\delta} &&4 \edge[dll]_(0.2){\delta} \edge[d]^{\gamma} &&2 \edge[dr]_{\alpha} &&2 \edge[dl]_{\beta} &5 \edge[dll]^{\tau}  \\
4 \edge[dr]_{\gamma} &&4 \edge[dl]^{\delta} &&3 \edge[dr]_{\nu} &&3 \edge[dl]^{\nu} &&&4  \\
&3 &&&&5  \\
&&4 \edge[dl]_{\gamma} \edge[drr]^{\delta} &&&&&&5 \edge[d]^{\tau}  \\
&3 \edge[dl]_{\rho} \edge[d]^(0.7){\sigma} \edge[dr]^{\nu} &&&3 \edge[dl]_{\rho} \edge[d]^(0.7){\sigma} \edge[dr]^{\nu} &&&&4 \edge[d]^{\gamma}  \\
2 \edge[dr]_{\alpha} &2 \edge[d]^(0.3){\beta} &5 \edge[dl]^{\tau} &2 \edge[dr]_{\alpha} &2 \edge[d]^(0.3){\beta} &5 \edge[dl]^{\tau} &&&3 \edge[dl]_{\rho} \edge[d]^(0.7){\sigma} \edge[dr]^{\nu}  \\
&4 &&&4 &&&2 &2 &5
}$$
Then $\pinf(\lamod)$ is contravariantly finite in $\lamod$, and $\lamod$ allows for unlimited iteration of strong tilting.
\end{ejem}

\begin{proof}
First one checks that the simple left $\Lambda$-modules corresponding to the vertices $e_4$ and $e_5$ have finite projective dimension, namely $\pdim_{\Lambda} S_4 = 1$ and $\pdim_ {\Lambda}  S_5 = 3$.  Choose $e: = e_1 + e_2 + e_3$.   To check that $e \Lambda (1 - e)$ is projective in $e \Lambda  e\mbox{\rm-mod}$, observe that $\Lambda ' : = e \Lambda e$ is the algebra determined by the quiver $Q'$ and the indecomposable projective left $\Lambda '$-modules shown below. 
\smallskip
\smallskip

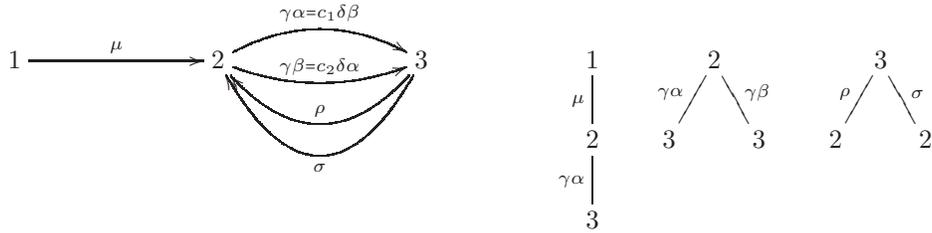
\centerline{$\xymatrixcolsep{5.5pc}
\xymatrix{
1 \ar[r]^{\mu} &2 \ar@/^1pc/[r]^{\gamma\alpha=c_1\delta\beta} \ar@/_0.7pc/[r]^{\gamma\beta=c_2\delta\alpha} &3 \ar@/^2pc/[l]_{\rho} \ar@/^3pc/[l]^{\sigma} 
}$ \qquad \qquad
$\xymatrixrowsep{1.5pc}\xymatrixcolsep{0.5pc}
\xymatrix{
1 \edge[d]_{\mu} &&&2 \edge[dl]_{\gamma\alpha} \edge[dr]^{\gamma\beta} &&&&3 \edge[dl]_{\rho} \edge[dr]^{\sigma}  \\
2 \edge[d]_{\gamma\alpha} &&3 &&3 &&2 &&2  \\
3
}$}

\noindent We now read off that $e \la e_4 \cong ( \la' e_3)^2$ and $e \la e_5 \cong \la' e_3$ in
$ \Lambda ' \mbox{\rm-mod}$, to find that the left $\la'$-module $e \la (1-e)$ is indeed projective.  To see that $\pinf( \Lambda ' \mbox{\rm-mod})$ is contravariantly finite in $ \Lambda ' \mbox{\rm-mod}$, we check that $\lfindim \la' = 0$ by applying Bass's criterion [4]:  Indeed, the graphs of the indecomposable projective right $\la'$-modules,
$$\xymatrixrowsep{1.5pc}\xymatrixcolsep{1pc}
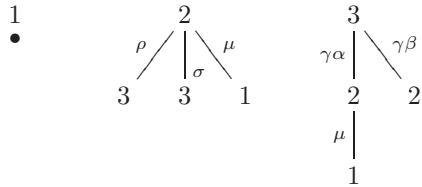
\xymatrix{
1 \drbl &&&2 \edge[dl]_{\rho} \edge[d]^(0.7){\sigma} \edge[dr]^{\mu} &&&3 \edge[d]_{\gamma\alpha} \edge[dr]^{\gamma\beta}  \\
&&3 &3 &1 &&2 \edge[d]_{\mu} &2  \\
&& && &&1
}$$
\noindent  show that all simple right $\la'$-modules embed into the right socle of $\la'$.  Consequently, Theorem  \ref{thm.reduction-to-simples-of-infteprojdim}
 yields contravariant finiteness of 
$\pinf({\lamod})$ in
 $\lamod$.  
 
 To confirm that $\lamod$ even allows for unlimited iteration of strong tilting, we show that this is true for $\la'$-mod, whence Corollary \ref{cor.iteration-stongtilting} will yield our claim.  Since $\lfindim \la' = 0$, the basic strong tilting object in $ \Lambda ' \mbox{\rm-mod}$ is ${_{\la'} T'} =  
{_{\la'} \la'}$, which yields $\widetilde{\la'} \cong \la'$.  

It is not difficult to directly ascertain that
 %$\pinf(\text{mod-}\la')$ % 
 $\pinf (\mbox{\rm mod-}\Lambda ') $
 is 
in turn contravariantly finite in $\mbox{\rm mod-}\Lambda '$, but another application of Theorem \ref{thm.reduction-to-simples-of-infteprojdim} cuts this task short:  Setting $e' = e'_2 + e'_3$, where the $e'_i$ are the primitive idempotents of $
\la'$ corresponding to the vertices of $Q'$, it is effortless to check that $e'$ satisfies the conditions of Setting 4.1 
relative to $\mbox{\rm mod-}\Lambda '$, and that 
$\rfindim e'\la'e' = 0$.  Hence $\pinf( \mbox{\rm mod-} e' \la' e')$ is contravariantly finite in
 $\mbox{\rm mod-} e' \la' e'$, and consequently 
so is $\pinf(\mbox{\rm mod-}\Lambda ')$ in $\mbox{\rm mod-}\Lambda '$.   
By Theorem \ref{teor.strong-tilting-iteration}, we thus conclude that $ \Lambda ' \mbox{\rm-mod}$  allows for 
unlimited iteration of strong tilting.  This completes the argument.  
\end{proof}

\end{document}